\documentclass[a4paper,12pt]{article}

\usepackage[margin=1.2in]{geometry}

\usepackage{amsmath}
\usepackage{algorithmic}
\usepackage{algorithm}
\usepackage{amsthm}
\usepackage{amsfonts}
\usepackage{comment}
\usepackage{graphicx}
\usepackage{subcaption}
\usepackage{hyperref}
\usepackage{color}

\newtheorem{assumption} {Assumption}
\newtheorem{theorem} {Theorem}
\newtheorem{lemma} {Lemma}
\newtheorem{definition} {Definition}
\newtheorem{corollary} {Corollary}

\def\x{{\mathbf{x}}}

\def\u{{\mathbf{u}}}

\def\v{{\mathbf{v}}}
\def\z{{\mathbf{z}}}
\def\w{{\mathbf{w}}}
\def\y{{\mathbf{y}}}
\def\p{{\mathbf{p}}}
\def\b{{\mathbf{b}}}

\def\A{{\mathbf{A}}}

\def\I{{\mathbf{I}}}

\def\H{{\mathbf{H}}}
\def\x{{\mathbf{x}}}
\def\V{{\mathbf{V}}}

\def\U{{\mathbf{U}}}

\newcommand{\phir}{\phi^r}

\newcommand{\mL}{\mathcal{L}}
\newcommand{\mA}{\mathcal{A}}
\newcommand{\R}{\mathcal{R}}
\newcommand{\mV}{\mathcal{V}}
\newcommand{\mP}{\mathcal{P}}

\newcommand{\mF}{\mathcal{F}}

\newcommand{\mK}{\mathcal{K}}

\newcommand{\mM}{\mathcal{M}}

\newcommand{\E}{\mathbb{E}}

\newcommand{\dom}{\textrm{dom}}

\newcommand{\diag}{\textrm{diag}}
\newcommand{\nnz}{\textrm{nnz}}

\newcommand{\rank}{\textrm{rank}}
\newcommand{\reals}{\mathbb{R}}

\title{Weak Proximal Newton Oracles  \\for  Composite  Convex Optimization}
\date{}
\author{Dan Garber \\ \small{dangar@technion.ac.il} \\ \small{Faculty of Data and Decision Sciences} \\ \small{Technion - Israel Institute of Technology}}

\begin{document}

 \maketitle
 
\begin{abstract}
Second-order methods are of great importance for composite convex optimization problems due to their local super-linear convergence rates (under appropriate assumptions). However, the presence of even a simple nonsmooth function in the model most often renders the subproblems in proximal Newton methods computationally difficult to solve in high dimensions. We introduce a novel approach based on a \textit{weak proximal Newton oracle} (WPNO), which only requires solving such subproblems to accuracy that is comparable to that of the \emph{optimal solution of the global problem}, while maintaining local super-linear convergence under standard assumptions. Mainly, unlike classical inexact proximal Newton schemes, the complexity of our WPNO is not tied to (approximately) minimizing each subproblem; instead, we establish that when the optimal solution of the global problem admits a sparse structure, the inner subproblem can be solved by specialized first-order methods whose cost scales directly with the sparsity of this solution rather than with the ambient dimension.
\end{abstract} 
 
\section{Introduction}
We consider the following standard composite convex optimization problem:
\begin{align}\label{eq:optprob}
\min_{\x\in\E}\{F(\x) := f(\x) + \R(\x)\}, 
\end{align}
where $\E$ is a Euclidean vector space of dimension $n$, $\R:\E\rightarrow(-\infty,\infty]$ is a proper closed convex function, and $f(\cdot)$ is convex and twice continuously differentiable over $\dom(\R)$. We let $\beta_2$ denote the Lipschitz continuity parameter of the Hessian of $f$,  i.e., $\Vert{\nabla^2f(\x) - \nabla^2f(\y)}\Vert \leq \beta_2\Vert{\x-\y}\Vert$ for all $\x,\y\in\dom(\R)$, where  $\Vert{\cdot}\Vert$ denotes the Euclidean norm. Throughout we shall assume that Problem \eqref{eq:optprob} is bounded from below and we shall denote the optimal value by $F^*$. Throughout we shall also assume either one of the following two standard assumptions holds true.

\begin{assumption}\label{ass:sc}
Given an initialization point $\x_1\in\dom(\R)$, there exists positive scalars $\alpha,\beta$ such that for every $\x$ in the initial level set $\mL_1:=\{\x\in\E~|~F(\x) \leq F(\x_1)\}$ it holds that $\beta\I \succeq \nabla{}^2f(\x) \succeq \alpha\I$. 
\end{assumption}
Note Assumption \ref{ass:sc} implies that there is  a unique minimizer $\x^*$ to Problem \eqref{eq:optprob}.

\begin{assumption}\label{ass:qg}
The function $\R$ has bounded domain with Euclidean diameter $D < \infty$, and the quadratic growth property holds with some constant $\alpha >0$, i.e., 
\begin{align*}
\forall \x\in\dom(\R): \quad \Vert{\x-\x^*}\Vert^2 \leq \frac{2}{\alpha}\left({F(\x) - F^*}\right),
\end{align*}
where $\x^*$ is the unique minimizer of Problem \eqref{eq:optprob}\footnote{since our derivations will be often interesting when $\R$ is a sparsity-promoting function, and when the optimal solutions are indeed sparse in the appropriate sense, the assumption here that $\x^*$ is unique is quite plausible. Nevertheless, this assumption is virtually without loss of generality, as the extension of the analyzes to handle multiple solutions is quite straightforward}.
\end{assumption}

In this paper we will be interested in efficient second-order methods for Problem \eqref{eq:optprob} with local super-linear convergence rates (under either Assumption \ref{ass:sc} or \ref{ass:qg}), which could be highly beneficial, at least at a certain proximity of the optimal solution, when Problem \eqref{eq:optprob} is ill-conditioned for first-order methods in the sense that the first-order condition number $\beta/\alpha$ is very large, where $\beta$ denotes the Lipschitz continuity parameter of $\nabla{}f$. 

The literature on second-order methods for convex optimization is of course huge, and we do not presume to thoroughly survey it here. Instead we refer the interested reader to \cite{lee2014proximal, patriksson1998cost, nesterov2006cubic, doikov2020convex, nesterov2008accelerating,  byrd2016inexact, li2017inexact, schmidt201211, boyd2004convex} and references therein.

Standard Newton methods, when applied to Problem \eqref{eq:optprob}, often require on each iteration $t$ to solve a subproblem of the form:
\begin{align}\label{eq:newtonStep}
\min_{\w\in\E}\{\Psi_t(\w) := \langle{\w-\x_t, \nabla{}f(\x_t)}\rangle + \frac{\eta}{2}\langle{\w-\x_t, \nabla^2f(\x_t)(\w-\x_t)}\rangle + \R(\w)\},
\end{align} 
where $\x_t$ is the current iterate, $\eta > 0$ is a step-size, and $\langle{\cdot,\cdot}\rangle$ is the standard inner-product, see for instance \cite{lee2014proximal, patriksson1998cost, doikov2020convex, schmidt201211}. Some methods consider a cubic-regularized variant of \eqref{eq:newtonStep}, see for instance \cite{nesterov2006cubic, nesterov2006cubicCon}.

Problem  \eqref{eq:newtonStep} is most often difficult to solve even when $\R$ is quite simple. For instance, already when $f$ is non-linear and $\R$ is the indicator function for an $\ell_1$ ball in $\reals^n$, Problem \eqref{eq:newtonStep}  does not admit a closed-form solution and requires the use of iterative optimization methods, e.g., first-order methods, to solve it to sufficient accuracy, within each iteration of the Newton method. Indeed there is a significant body of work on  \emph{inexact} proximal Newton methods that maintain a local super-linear convergence rate, provided that Problem \eqref{eq:newtonStep} is solved to error that is driven sufficiently fast (with the Newton iterations counter $t$) to zero, see for instance \cite{lee2014proximal, byrd2016inexact, li2017inexact} and references therein. 

Unfortunately, solving Problem  \eqref{eq:newtonStep} within inexact Newton methods using standard first-order methods can often defeat the purpose of using a second-order method. For instance, even an optimal  first-order method with an accelerated linear convergence rate ,e.g., \cite{nesterov1983method, beck2009fast, nesterov2013gradient}, will require in worst-case  number of gradient computations of (the smooth part of) $\Psi_t(\cdot)$ and prox computations (e.g., Euclidean projections) w.r.t. $\R$ that scale with $\sqrt{\beta/\alpha}$. Computing the gradient of $\Psi_t$ requires in general a dense matrix-vector product computation (multiplication with Hessian matrix), and in case $\R$ is for instance the indicator function for a nuclear norm ball of matrices or a non-trivial polytope, computing the prox operator may also be computationally difficult. Thus, the inner complexity bounds of such inexact Newton methods  are typically expressed in terms of the condition number of $\Psi_t$ and the ambient dimension, without taking advantage of any sparsity, low-rankness, or low-face structure of the global optimum $\x^*$.



In case the bottleneck in solving \eqref{eq:newtonStep} via proximal gradient methods is the complex structure of $\R$, some works such as \cite{carderera2020second, liu2022newton} have considered using the Frank-Wolfe (aka conditional gradient) method, which only requires a linear minimization oracle w.r.t. $\dom(\R)$. However, the slow convergence rate of the Frank-Wolfe method, which is typically $1/k$,  may again result in an overall prohibitive runtime, in particular when the overall desired accuracy is medium-high. In case $\R$ is the indicator function for a polytope, some variants of Frank-Wolfe enjoy a linear convergence rate, however this comes, in worst-case, with a multiplicative factor that scales at least linearly with the product of the first-order condition number of $\Psi_t$ and the ambient dimension \cite{carderera2020second}. We note that while some Frank-Wolfe variants for polytopes can replace the explicit dependence on the ambient dimension with the dimension of the optimal face (containing the optimal solution), which can indeed be low-dimensional in certain settings \cite{garber2016linear, bashiri2017decomposition}, even if $\x^*$ indeed lies on a low-dimensional face, there is in general no guarantee that the solutions to the subproblems \eqref{eq:newtonStep} will also lie on a low-dimensional face.

The above discussions lead us to the following conceptual question:
\begin{center}
\textit{In case the optimal solution to Problem \eqref{eq:optprob} admits some known sparse structure, can it be leveraged towards provable more efficient implementations of proximal Newton methods, without sacrificing the local super-linear convergence?}
\end{center}

In this work we develop a new principled approach towards making proximal Newton methods more efficient to implement while maintaining local super-linear convergence rates. Indeed our approach will be mostly of interest in case the optimal solution $\x^*$ admits a certain sparse structure. For example, when $\R$ is the indicator function for an $\ell_1$ ball or a matrix nuclear norm ball, or even a polytope, we may often expect that $\x^*$ is indeed sparse, may it be entry-wise sparsity for the $\ell_1$ ball case, low rank for the matrix nuclear norm ball case, or that it lies on a low-dimensional face in the polytope case.  

Our approach is based on two key observations. First, we observe that in order to obtain local super-linear convergence (in function value and under either Assumption \ref{ass:sc} or Assumption \ref{ass:qg}), subproblem \eqref{eq:newtonStep} (or a regularized variant of it) need not be solved to optimality (or even with very good accuracy as in classical inexact proximal Newton analyzes \cite{lee2014proximal, byrd2016inexact, li2017inexact}), but only to value that is comparable with that of $\x^*$ --- the optimal solution to our original Problem \eqref{eq:optprob}. We refer to these as \textit{weak proximal Newton oracles} (Definition \ref{def:2wpo} in the sequel). Second, and perhaps surprisingly, we are able to show that \emph{specialized} first-order methods, when applied to subproblems such as  \eqref{eq:newtonStep}, but require to solve it only to the level of error achieved by $\x^*$, can solve these with complexity that scales directly with the sparsity of $\x^*$ and not with the ambient dimension as in standard first-order methods, despite the fact that in general the optimal solutions to subproblem \eqref{eq:newtonStep} need not inherit the sparse structure of $\x^*$. For instance, if $\R$ is the indicator for an $\ell_1$-ball and $\x^*$ is (as often expected in this case) entry-wise sparse, then our solvers for subproblem \eqref{eq:newtonStep} will require to compute the product of the Hessian with only a sparse vector (with same level of sparsity as assumed for $\x^*$), instead of a dense vector as in standard gradient methods. In case $\R$ is the indicator for a matrix nuclear norm ball and $\x^*$ is low-rank (again, as often expected when using such $\R$), our solvers will require to compute only a low-rank singular value decomposition (SVD) (with same rank as the one assumed for $\x^*$), instead of a full-rank SVD required for the projection step onto the nuclear norm ball in standard gradient methods, which is computationally prohibitive in high dimensions. In case $\R$ is a polytope suitable for Frank-Wolfe/conditional  gradient methods (i.e., implementing the linear minimization oracle is efficient), our conditional gradient-based solver will converge linearly and with number of iterations that scales with the dimension of the optimal face containing $\x^*$ and not with the ambient dimension as in standard linearly-converging conditional gradient methods. The specialized first-order methods which we develop are based on \textit{first-order weak proximal oracles}, recently introduced in \cite{pmlr-v89-garber19a, garber2023faster}, and \textit{decomposition-invariant} conditional gradient methods \cite{garber2016linear, bashiri2017decomposition}. 

We note that while a significant part of the literature on second-order methods focuses on \textit{quasi-Netwon} methods in which the second derivative in subproblem \eqref{eq:newtonStep} is replaced with certain approximation $\H_t\approx\nabla^2f(\x_t)$ which is computationally more efficient to work with \cite{lee2014proximal}, such developments are in a sense orthogonal to the approach proposed here and are beyond the scope of this current work. It can indeed be an interesting future direction to examine whether the combination of the ideas presented here and quasi-Newton methods could lead to even more efficient methods.

\subsection{Organization of this paper}
In section \ref{sec:wpo} we define our main object of interest --- \textit{weak proximal Newton oracles} (WPNO), and establish their local super-linear convergence. In Section \ref{sec:fowpo} we derive efficient implementations of our WPNOs based on \emph{first-order} weak proximal oracles. In this section we also discuss concrete examples where our approach may significantly improve over standard methods. In section \ref{sec:dicgwpo} we discuss efficient implementations of our WPNOs based on a state-of-the-art conditional gradient method. 
Finally, in Section \ref{sec:exp} we bring empirical evidence to demonstrate the potential practical benefits of our results.

\section{Weak Proximal Newton Oracles}\label{sec:wpo} 
\begin{definition}[weak proximal Newton oracles]\label{def:2wpo}
We say a map $\mA:\dom(\R)\times\reals_+\rightarrow\dom(\R)$ is a (approximated) weak proximal Newton oracle for Problem \eqref{eq:optprob} with some parameter $C_{\mA}\geq 0$, abbreviated as WPNO, if given input $(\x,\eta)\in\dom(\R)\times\reals_+$, it outputs a point $\v\in\dom(\R)$ such that 
\begin{align}\label{eq:wpno}
\phi_{\x,\eta}(\v) \leq \phi_{\x,\eta}(\x^*) + \eta^2C_{\mA}\Vert{\x-\x^*}\Vert^3,
\end{align}
where
\begin{align}\label{eq:wpno:func}
\phi_{\x,\eta}(\w):=\langle{\w-\x, \nabla{}f(\x)}\rangle + \frac{\eta}{2}\langle{\w-\x, \nabla^2f(\x)(\w-\x)}\rangle + \R(\w).
\end{align}
We say a map $\mA:\dom(\R)\times\reals_+\rightarrow\dom(\R)$ is a (approximated) weak \textit{\underline{regularized}} proximal Newton oracle for Problem \eqref{eq:optprob} with some parameter $C_{\mA}\geq 0$, abbreviated as WRPNO, if given input $(\x,\eta)\in\dom(\R)\times\reals_+$, it outputs a point $\v\in\dom(\R)$ such that 
\begin{align}\label{eq:wrpno}
\phir_{\x,\eta}(\v) \leq \phir_{\x,\eta}(\x^*) + \eta^2C_{\mA}\Vert{\x - \x^*}\Vert^3,
\end{align}
where
\begin{align}\label{eq:wrpno:func}
\phir_{\x,\eta}(\w):=\phi_{\x,\eta}(\w) + \frac{\eta^2\beta_2}{6}\Vert{\w-\x}\Vert^3.
\end{align}
\end{definition}
In other words and less formally: A WPNO solves the same type of subproblem as in the proximal Newton method (e.g., \cite{lee2014proximal}), but does so only to an approximation error competitive with the global optimal solution $\x^*$, i.e., up to an additive error scaling with $\Vert{\x-\x^*}\Vert^3$ (where $\x$ is the current point) from $\x^*$. A WRPNO does the same but w.r.t. the type of subproblem appearing in the cubic-regularized Newton method (e.g. \cite{nesterov2006cubic}).
 
Let us comment why we consider both regularized and unregularized WPNOs. As we shall see in Theorem \ref{thm:main}, the use of the regularized WRPNO will lead to a preferable super-linear convergence rate, however it will mostly be interesting when Assumption \ref{ass:qg} holds. This is because under Assumption \ref{ass:qg}, the function $\phi_{\x,\eta}^r$ admits a bounded Lipschitz parameter for the gradient of its smooth part over $\dom(\R)$ (since $\dom(\R)$ itself is bounded), which is compatible with the use of first-order methods to implement the WRPNO. When only Assumption \ref{ass:sc} holds, we can no longer assume the smooth part of $\phi_{\x,\eta}^r$ has bounded Lipschitz parameter for its gradient (due to the cubic term), and instead we rely on its unregularized  version $\phi_{\x,\eta}$, which is in particular quadratic, but which will lead to a slightly inferior (in terms of constants) super-linear convergence rate.

We introduce some useful notation. 
For any $\x\in\dom(\R)$ and $\z\in\E$ we denote $\Vert{\z}\Vert_{\x} = \sqrt{\z^{\top}\nabla^2{}f(\x)\z}$. Throughout we shall denote the differential parts of $\phi_{\x,\eta}(\cdot)$ and $\phi_{\x,\eta}^r(\cdot)$ as:
\begin{align*}
Q_{\x,\eta}(\w) &:= \langle{\w-\x, \nabla{}f(\x)}\rangle + \frac{\eta}{2}\langle{\w-\x, \nabla^2f(\x)(\w-\x)}\rangle; \\
Q_{\x,\eta}^r(\w) &:= \langle{\w-\x, \nabla{}f(\x)}\rangle + \frac{\eta}{2}\langle{\w-\x, \nabla^2f(\x)(\w-\x)}\rangle + \frac{\eta^2\beta_2}{6}\Vert{\w - \x}\Vert^3.
\end{align*}
In particular, when $\eta=1$ (unit step-size), we shall simply write $Q_{\x}(\w)$ and $Q_{\x}^r(\w)$, respectively.

We now present several lemmas that will be important for our derivations.
\begin{lemma}[Lemma 1 in \cite{nesterov2006cubic}]\label{lem:funie}
For any $\x,\y\in\dom(\R)$ it holds that
\begin{align*}
\left\vert{}f(\y) - f(\x) - \langle{\y-\x, \nabla{}f(\x)}\rangle - \frac{1}{2}\langle{\y-\x, \nabla^2f(\x)(\y-\x)}\rangle\right\vert \leq \frac{\beta_2}{6}\Vert{\x-\y}\Vert^3.
\end{align*}
\end{lemma}



Lemma \ref{lem:funie} leads to the following inequality which will be of use.
\begin{lemma}\label{lem:QGprox}
For any $\x,\y\in\dom(\R)$ it holds that,
\begin{align*}
\phir_{\x,1}(\y) - \phir_{\x,1}(\x^*) \geq F(\y) - F^*- \frac{\beta_2}{3}\Vert{\x-\x^*}\Vert^3.
\end{align*}
\end{lemma}
\begin{proof}
Fix some $\x, \y$ in $\dom(\R)$.
From Lemma \ref{lem:funie} we have the following two inequalities:
\begin{align*}
Q_{\x}^r(\x^*) &\leq f(\x^*) - f(\x) + \frac{\beta_2}{3}\Vert{\x-\x^*}\Vert^3 \quad \textrm{and} \quad  Q_{\x}^r(\y) \geq f(\y) - f(\x). 
\end{align*}
Combining, we have that
\begin{align*}
Q_{\x}^r(\y) - Q_{\x}^r(\x^*) \geq f(\y) - f(\x^*) -  \frac{\beta_2}{3}\Vert{\x-\x^*}\Vert^3.
\end{align*}
Adding $\R(\y)-\R(\x^*)$ to both sides completes the proof.
\end{proof}

The following lemma bounds the distance of the output of a WPNO (Definition \ref{def:wpo}) from the optimal solution $\x^*$. For clarity of presentation, we defer the proof to the appendix.
\begin{lemma}\label{lem:distBound}
Suppose Assumption \ref{ass:sc} holds and let $\x\in\mL_1$. Let $\v\in\dom(\R)$ be the output of a WPNO with approximation parameter $C_{\mA}$, when called with input $(\x,\eta)\in\dom(\R)\times\reals_+$. Then,
\begin{align*} 
\Vert{\v-\x^*}\Vert \leq \frac{\beta}{\alpha}\left({3+\frac{2}{\eta}}\right)\Vert{\x-\x^*}\Vert + \frac{\eta{}C_{\mA}}{\alpha}\Vert{\x-\x^*}\Vert^2.
\end{align*}
\end{lemma}

\subsection{Local super-linear convergence}
\begin{theorem}[local super-linear convergence with WPNOs and unit step-size]\label{thm:main}
Suppose either Assumption \ref{ass:sc} or Assumption \ref{ass:qg} holds true, and let $\mA$ be a \textit{weak \underline{regularized} proximal Newton oracle} with parameter $C_{\mA}$, as defined in Eq. \eqref{eq:wrpno}.  Consider a sequence of points $(\x_t)_{t\geq 1}$ such that $\x_1\in\dom(\R)$ and
\begin{align}\label{eq:updatestep}
\forall t\geq 1: \qquad \x_{t+1} \gets \left\{ \begin{array}{ll}
         \mA(\x_{t}, 1) & \mbox{if $F(\mA(\x_{t}, 1)) < F(\x_t)$};\\
        \x_t & else.\end{array} \right.
\end{align}
(i.e., using fixed step-size $\eta=1$). Then, 
\begin{align}\label{eq:supconv:1}
\forall t\geq 1: \qquad F(\x_{t+1}) - F^* \leq \frac{2\sqrt{2}}{\alpha^{3/2}}\left({C_{\mA}+\frac{\beta_2}{3}}\right)\left({F(\x_t) - F^*}\right)^{3/2}.
\end{align} 
If the oracle $\mA$ in \eqref{eq:updatestep} is replaced with a (non-regularized) \textit{weak proximal Newton oracle} with parameter $C_{\mA}$, as defined in Eq. \eqref{eq:wpno}, and Assumption \ref{ass:sc} holds true, then the resulting  sequence of points $(\x_t)_{t\geq 1}$ satisfies,
\begin{align}\label{eq:supconv:2}
\forall t\geq 1: ~~~ F(\x_{t+1})- F^* &\leq \frac{2\sqrt{2}}{\alpha^{3/2}}\left({\frac{5\beta_2}{6}+C_{\mA} + \frac{1000\beta_2\beta^3}{3\alpha^3}}\right)(F(\x_t)-F^*)^{3/2} \nonumber \\
&~~~+\frac{64\beta_2C_{\mA}^3}{3\alpha^{6}}\left({F(\x_t)-F^*}\right)^3.
\end{align}  
\end{theorem}
In both cases of the theorem, if $C_{\mA} = O(\beta_2)$, then the local super-linear convergence rate is of the form $F(\x_{t+1}) -F^* = O\left({\left({F(\x_t) - F^*}\right)^{3/2}}\right)$, provided that the intialization point $\x_1$ is such that the initial error $F(\x_1) - F^*$ is already sufficiently small. In particular,  the rate in \eqref{eq:supconv:1} is the same, up to a universal constant, as the rate obtained in \cite{nesterov2006cubic} (Theorem 5) for the \textit{exact} cubic-regularized Newton method in the \textit{unconstrained case} (i.e., $\R\equiv 0$). 

As already discussed above, while the rate in \eqref{eq:supconv:1} is clearly preferable to \eqref{eq:supconv:2}, it will be mostly of interest when Assumption \ref{ass:qg} holds (since then the regularized WPNO could be implemented efficiently), while the rate in \eqref{eq:supconv:2} will be relevant under Assumption \ref{ass:sc}.

\begin{proof}[Proof of Theorem \ref{thm:main}]
Using Lemma \ref{lem:funie} we have that for any $t\geq 1$,
\begin{align}\label{eq:thm:suplin:1}
F(\x_{t+1}) &= f(\x_{t+1}) + \R(\x_{t+1})  \nonumber \\
&\leq f(\x_t) + \langle{\x_{t+1}-\x_t,\nabla{}f(\x_t)}\rangle + \frac{1}{2}\langle{\x_{t+1}-\x_t,\nabla^2f(\x_t)(\x_{t+1}-\x_t)}\rangle \nonumber\\
&~~~ + \frac{\beta_2}{6}\Vert{\x_{t+1}-\x_t}\Vert^3 + \R(\x_{t+1}).
\end{align}
Using the definition of $\phir_{\x_t,1}$ (Eq. \eqref{eq:wrpno:func}) and the definition of $\x_{t+1}$ in \eqref{eq:updatestep} we have that,
\begin{align*}
F(\x_{t+1})&\leq f(\x_t) + \phir_{\x_t,1}(\x_{t+1}) \leq f(\x_t) + \phir_{\x_t,1}(\x^*) + C_{\mA}\Vert{\x_t-\x^*}\Vert^3.
\end{align*}
Using Lemma \ref{lem:funie} again we have that,
\begin{align*}
\phir_{\x_t,1}(\x^*) \leq f(\x^*) -f(\x_t) +  \frac{\beta_2}{3}\Vert{\x^*-\x_t}\Vert^2 + \R(\x^*).
\end{align*}
Combining the last two inequalities  yields,
\begin{align*}
F(\x_{t+1}) \leq F^*  + \left({C_{\mA} + \frac{\beta_2}{3}}\right)\Vert{\x_t-\x^*}\Vert^3.
\end{align*}
Both Assumption \ref{ass:sc} and Assumption \ref{ass:qg} imply that $\Vert{\x_t-\x^*}\Vert^2 \leq \frac{2}{\alpha}\left({F(\x_t)-F^*}\right)$. Using this, and subtracting $F^*$ from both sides, we have that indeed
\begin{align*}
F(\x_{t+1})-F^* \leq \frac{2\sqrt{2}}{\alpha^{3/2}}\left({C_{\mA} + \frac{\beta_2}{3}}\right)\left({F(\x_t) - F^*}\right)^{3/2},
\end{align*}
which proves the first part of the theorem.\\

Let us now prove in similar fashion the second part of the theorem under Assumption \ref{ass:sc}. Starting from Eq. \eqref{eq:thm:suplin:1} and using the definition of $\phi_{\x_t,1}$ (Eq. \eqref{eq:wpno:func}) and the definition of $\x_{t+1}$ we have that,
\begin{align*}
F(\x_{t+1})&\leq f(\x_t) + \phi_{\x_t,1}(\x_{t+1}) + \frac{\beta_2}{6}\Vert{\x_{t+1}-\x_t}\Vert^3 \\
&\leq f(\x_t) + \phi_{\x_t,1}(\x^*) +  \frac{\beta_2}{6}\Vert{\x_{t+1}-\x_t}\Vert^3 + C_{\mA}\Vert{\x_t-\x^*}\Vert^3 \\
&\leq f(\x_t) + \phi_{\x_t,1}(\x^*) +  \frac{\beta_2}{6}\left({\Vert{\x_{t+1}-\x^*}\Vert + \Vert{\x_t-\x^*}\Vert}\right)^3 + C_{\mA}\Vert{\x_t-\x^*}\Vert^3 \\
&\leq f(\x_t) + \phi_{\x_t,1}(\x^*) +  \frac{2\beta_2}{3}\Vert{\x_{t+1}-\x^*}\Vert^3 + \left({\frac{2\beta_2}{3}+C_{\mA}}\right)\Vert{\x_t-\x^*}\Vert^3,
\end{align*}
where the last inequality follows the inequality $(a+b)^3 \leq 4(a^3+b^3)$ that holds for any two non-negative scalars $a,b$.

Using Lemma \ref{lem:funie} we have that,
\begin{align*}
\phi_{\x_t,1}(\x^*)\leq f(\x^*) -f(\x_t) +  \frac{\beta_2}{6}\Vert{\x^*-\x_t}\Vert^2 + \R(\x^*).
\end{align*}
Using this and Lemma \ref{lem:distBound} w.r.t. the term $\Vert{\x_{t+1}-\x^*}\Vert$,  we obtain that
\begin{align*}
F(\x_{t+1}) &\leq F^*  + \left({\frac{5\beta_2}{6}+C_{\mA}}\right)\Vert{\x_t-\x^*}\Vert^3 \\
&~~~+  \frac{2\beta_2}{3}\left({\frac{5\beta}{\alpha}\Vert{\x_t-\x^*}\Vert + \frac{C_{\mA}}{\alpha}\Vert{\x_t-\x^*}\Vert^2}\right)^3 \\
&\leq F^* + \left({\frac{5\beta_2}{6}+C_{\mA} + \frac{1000\beta_2\beta^3}{3\alpha^3}}\right)\Vert{\x_t-\x^*}\Vert^3 +\frac{8\beta_2C_{\mA}^3}{3\alpha^3}\Vert{\x_t-\x^*}\Vert^6,
\end{align*}
where in the last inequality we have used again the inequality $(a+b)^3 \leq 4(a^3+b^3)$.

Using the fact that under Assumption \ref{ass:sc}, $F(\cdot)$ is $\alpha$-strongly convex over the initial level set $\mL_1$, and subtracting $F^*$ from both sides we have that,
\begin{align*}
F(\x_{t+1})- F^* &\leq \frac{2\sqrt{2}}{\alpha^{3/2}}\left({\frac{5\beta_2}{6}+C_{\mA} + \frac{1000\beta_2\beta^3}{3\alpha^3}}\right)(F(\x_t)-F^*)^{3/2} \\
&~~~+\frac{64\beta_2C_{\mA}^3}{3\alpha^{6}}\left({F(\x_t)-F^*}\right)^3,
\end{align*}
which proves the second part of the theorem.
\end{proof}

\section{Implementing WPNOs via First-Order Weak Proximal Oracles}\label{sec:fowpo}
In this section we discuss the efficient implementation of WPNOs (Definition \ref{def:2wpo}) using specialized first-order methods based on  \textit{first-order weak proximal oracles} \cite{pmlr-v89-garber19a, garber2023faster}. In particular, we shall derive novel convergence results for these first-order methods that allow to leverage the sparse structure of $\x^*$, even when applied to optimization problems whose optimal solutions do not necessarily admit such sparse structure.

Throughout this section we fix some iteration $t$ of the iterations described in Eq. \eqref{eq:updatestep}. Accordingly, in case Assumption \ref{ass:sc} holds we shall use the short notation $Q_t = Q_{\x_t}$ and $\phi_t = \phi_{\x_t,1}$, and in case Assumption \ref{ass:qg} holds, we shall denote  $Q_t = Q_{\x_t}^r$ and $\phi_t = \phi_{\x_t,1}^r$. 
In particular, under Assumption \ref{ass:sc}, $Q_t$ is quadratic, and under Assumption \ref{ass:qg} $\dom(\R)$ is bounded, and thus we can conclude that in both cases $Q_t$ has Lipschitz continuous gradient over $\dom(\R)$, which we denote throughout this section by $\tilde{\beta}$. If $\phi_t = \phi_{\x_t,1}$ then we have $\nabla^2\phi_t = \nabla^2f(\x_t)$ and thus according to Assumption \ref{ass:sc} we can set $\tilde{\beta} = \beta$. If $\phi_t = \phir_{\x_t,1}$, then  a simple calculation yields that $\tilde{\beta} \leq \beta + \frac{\beta_2{}D}{2}$. That is, we can set
\begin{align}\label{eq:tbeta}
\tilde{\beta} =  \left\{\begin{array}{ll}
        \beta  & \mbox{if Assumption \ref{ass:sc} holds};\\
        \beta + \frac{\beta_2{}D}{2} & \mbox{if Assumption \ref{ass:qg} holds},\end{array} \right.
\end{align}


\begin{definition}[first-order weak proximal oracle]\label{def:wpo}
We say a map $\mA_1:\dom(\R)\times\reals_+\rightarrow\dom(\R)$ is a (first-order) weak proximal oracle (for $\phi_t$), abbreviated as WPO, if given inputs $(\y,\lambda)\in\dom(\R)\times\reals_+$, it outputs a point $\w\in\dom(\R)$ such that
\begin{align}\label{eq:wpo}
\psi_{\y,\lambda}(\w) \leq \min\{\psi_{\y,\lambda}(\x^*),~\psi_{\y,\lambda}(\y)\},
\end{align}
where
\begin{align}\label{eq:wpo:func}
\psi_{\y,\lambda}(\z):=\langle{\z-\y, \nabla{}Q_t(\y)}\rangle + \frac{\lambda\tilde{\beta}}{2}\Vert{\z-\y}\Vert^2 + \R(\z). 
\end{align}
\end{definition}
In Section \ref{sec:fowpo:imp} we discuss in detail cases of interest in which a first-order weak proximal oracle could be implemented very efficiently by leveraging sparse structures in $\x^*$. 

\begin{algorithm}
\begin{algorithmic}
\caption{W(R)PNO via First-Order WPO}\label{alg:wpo:1}
\STATE input: $\mA_1$ --- a first-order WPO (Definition \ref{def:wpo}), $\x_t\in\dom(\R)$ --- init. point
\STATE $\y_1 \gets \x_t$
\FOR{$i=1,2,\dots$}
\STATE choose step-size $\lambda_i\in[0,1]$
\STATE $\w_i \gets$ output of  $\mA_1$ when called with input $(\y_i,\lambda_i)$ 
\STATE $\y_{i+1} \gets (1-\lambda_i)\y_i + \lambda_i\w_i$
\ENDFOR
\end{algorithmic}
\end{algorithm}

\begin{theorem}
\label{thm:wpo:1}
If Assumption \ref{ass:sc} holds and $\phi_t = \phi_{\x_t,1}$, then the sequence $(\y_i)_{i\geq 1}$ produced by Algorithm \ref{alg:wpo:1} with a fixed step-size $\lambda_i = \frac{\alpha}{\tilde{\beta}}$ ($\alpha$ as in Assumption \ref{ass:sc}) satisfies for all $i\geq 1$:
\begin{align}\label{eq:thm:wpo1:res1}
\phi_t(\y_{i}) - \phi_t(\x^*) 
&\leq  \left({F(\x_t) - F^* + \frac{\beta_2}{6}\Vert{\x_t-\x^*}\Vert^3}\right)\left({1-\frac{\alpha}{\tilde{\beta}}}\right)^{i-1}.
\end{align}
If Assumption \ref{ass:qg} holds and $\phi_t = \phir_{\x_t,1}$, then using a fixed step-size $\lambda_i = \frac{\alpha}{2\tilde{\beta}}$ ( $\alpha$ as in Assumption \ref{ass:qg}) satisfies that for all $i\geq 1$:
\begin{align}\label{eq:thm:wpo1:res2}
\phi_t(\y_{i}) - \phi_t(\x^*) 
&\leq \left({F(\x_t) - F^*}\right)\left({1-\frac{\alpha}{4\tilde{\beta}}}\right)^{i-1} + \frac{\beta_2}{3}\Vert{\x_t-\x^*}\Vert^3.
\end{align}
\end{theorem}
Thus, we can conclude that in both scenarios (Assumption \ref{ass:sc} or Assumption \ref{ass:qg}), by running Algorithm \ref{alg:wpo:1} for sufficiently-many iterations ($\tilde{O}(\tilde{\beta}/\alpha)$, where $\tilde{O}(\cdot)$ suppresses logarithmic factors), we can implement a W(R)PNO with parameter $C_{\mA} = O(\beta_2)$.

\begin{proof}[Proof of Theorem \ref{thm:wpo:1}]
Fix some iteration $i$ of Algorithm \ref{alg:wpo:1}. Since $Q_t(\cdot)$ is $\tilde{\beta}$-smooth and $\R(\cdot)$ is convex, we have that
\begin{align}\label{eq:thm:wpo:1}
\phi_t(\y_{i+1}) &= Q_t((1-\lambda_i)\y_i + \lambda_i\w_i) + \R((1-\lambda_i)\y_i + \lambda_i\w_i) \nonumber\\
&\leq Q_t(\y_i) + \lambda_i\langle{\w_i - \y_i, \nabla{}Q_t(\y_i)}\rangle + \frac{\lambda_i^2\tilde{\beta}}{2}\Vert{\w_i- \y_i}\Vert^2 \nonumber\\
&~~~+ (1-\lambda_i)\R(\y_i) + \lambda_i\R(\w_i)\nonumber \\
&= \phi_t(\y_i) - \lambda_i\R(\y_i) + \lambda_i\psi_{\y_i,\lambda_i}(\w_i).
\end{align}
By definition we have that $\psi_{\y_i,\lambda_i}(\w_i) \leq \psi_{\y_i,\lambda_i}(\y_i) = \R(\y_i)$, which implies that $\phi_t(\y_{i+1})  \leq  \phi_t(\y_i)$. In particular, if for some iteration $i_0$ we have that $\phi_t(\y_{i_0}) \leq \phi_t(\x^*)$ then also for all $i > i_0$ we have that $\phi_t(\y_{i}) \leq \phi_t(\x^*)$. Thus, for the remaining of the proof we shall thus assume  $\phi_t(\y_i) > \phi_t(\x^*)$. 

Continuing from the RHS of \eqref{eq:thm:wpo:1} and using the definition of $\w_i$, we have that
\begin{align}\label{eq:thm:wpo:2}
\phi_t(\y_{i+1})  
& \leq \phi_t(\y_i) - \lambda_i\R(\y_i)+ \lambda_i\psi_{\y_i,\lambda_i}(\x^*) \nonumber \\
&= \phi_t(\y_i) - \lambda_i\R(\y_i) + \lambda_i\langle{\x^* - \y_i, \nabla{}Q_t(\y_i)}\rangle + \frac{\lambda_i^2\tilde{\beta}}{2}\Vert{\x^*- \y_i}\Vert^2 + \lambda_i\R(\x^*).
\end{align}

We now consider two cases. If Assumption \ref{ass:sc} holds, then $Q_t(\cdot)$ is $\alpha$-strongly convex which implies that
\begin{align*}
\langle{\x^* - \y_i, \nabla{}Q_t(\y_i)}\rangle \leq Q_t(\x^*) - Q_t(\y_i) - \frac{\alpha}{2}\Vert{\x^*-\y_i}\Vert^2.
\end{align*}
Plugging this into \eqref{eq:thm:wpo:2} and subtracting $\phi_t(\x^*)$ from both sides, we obtain
\begin{align*}
\phi_t(\y_{i+1}) -\phi_t(\x^*) &\leq (1-\lambda_i)(\phi_t(\y_i) - \phi_t(\x^*)) + \frac{\lambda_i}{2}\left({\lambda_i\tilde{\beta} - \alpha}\right)\Vert{\x^*-\y_i}\Vert^2.
\end{align*}
In particular, for $\lambda_i = \alpha/\tilde{\beta}$ we obtain
\begin{align*}
\phi_t(\y_{i+1}) -\phi_t(\x^*) &\leq \left({1-\frac{\alpha}{\tilde{\beta}}}\right)(\phi_t(\y_i) - \phi_t(\x^*)).
\end{align*}
Result \eqref{eq:thm:wpo1:res1} now follows by recalling that, under Assumption \ref{ass:sc}, we have $\phi_t = \phi_{\x_t,1}$ and so,
\begin{align*}
\phi_t(\y_1) - \phi_t(\x^*) &=  Q_t(\x_t) + \R(\x_t) - Q_t(\x^*) - \R(\x^*) \\
&= \R(\x_t) - Q_t(\x^*) - \R(\x^*) \\
&\leq f(\x_t) - f(\x^*) + \frac{\beta_2}{6}\Vert{\x_t-\x^*}\Vert^3 + \R(\x_t) - \R(\x^*) \\
&= F(\x_t) - F^* + \frac{\beta_2}{6}\Vert{\x_t-\x^*}\Vert^3,
\end{align*}
where the inequality follows from Lemma \ref{lem:funie}.\\

Now we turn to consider the case that Assumption \ref{ass:qg} holds and to prove result \eqref{eq:thm:wpo1:res2}. We proceed from Eq. \eqref{eq:thm:wpo:2} via a slightly different route. Using the convexity of $Q_t(\cdot)$ and the quadratic growth of $F(\cdot)$, we have that
\begin{align*}
\phi_t(\y_{i+1})  &\leq  \phi_t(\y_i) -\lambda_i\R(\y_i) + \lambda_i\left({Q_t(\x^*)-Q_t(\y_i)}\right) + \frac{\lambda_i^2\tilde{\beta}}{\alpha}\left({F(\y_i) -F^*}\right) + \lambda_i\R(\x^*) \\
&=\phi_t(\y_i) + \lambda_i(\phi_t(\x^*) - \phi_t(\y_i)) + \frac{\lambda_i^2\tilde{\beta}}{\alpha}\left({F(\y_i) -F^*}\right). 
\end{align*}

Recall that under Assumption \ref{ass:qg} we have that $\phi_t = \phir_{\x_t,1}$. Thus, by  Lemma \ref{lem:QGprox} we have that,
\begin{align*}
\phi_t(\y_{i+1})  &\leq  \phi_t(\y_i) + \frac{\lambda_i}{2}\left({\phi_t(\x^*)-\phi_t(\y_i)}\right)  + \frac{\lambda_i}{2}\left({\phi_t(\x^*)-\phi_t(\y_i)}\right) + \frac{\lambda_i^2\tilde{\beta}}{\alpha}\left({F(\y_i) -F^*}\right) \\
&\leq  \phi_t(\y_i) + \frac{\lambda_i}{2}\left({\phi_t(\x^*)-\phi_t(\y_i)}\right) \\
&+ \frac{\lambda_i\beta_2}{6}\Vert{\x-\x^*}\Vert^3 + \left({\frac{\lambda_i^2\tilde{\beta}}{\alpha} - \frac{\lambda_i}{2}}\right)\left({F(\y_i) - F^*}\right).
\end{align*}

Setting $\lambda_i = \frac{\alpha}{2\tilde{\beta}}$ and subtracting $\phi_t(\x^*)$ from both sides we obtain,
\begin{align*}
\phi_t(\y_{i+1}) - \phi_t(\x^*) &\leq \left({\phi_t(\y_i) - \phi_t(\x^*)}\right)\left({1- \frac{\alpha}{4\tilde{\beta}}}\right) + \frac{\alpha\beta_2}{12\tilde{\beta}}\Vert{\x_t-\x^*}\Vert^3\\
 &\leq \left({\phi_t(\y_1)-\phi_t(\x^*)}\right)\left({1- \frac{\alpha}{4\tilde{\beta}}}\right)^i  + \frac{1}{1-\left({1-\frac{\alpha}{4\tilde{\beta}}}\right)}\frac{\alpha\beta_2}{12\tilde{\beta}}\Vert{\x_t-\x^*}\Vert^3 \\
&= \left({\phi_t(\y_1)-\phi_t(\x^*)}\right)\left({1- \frac{\alpha}{4\tilde{\beta}}}\right)^i  + \frac{\beta_2}{3}\Vert{\x_t-\x^*}\Vert^3.
\end{align*}
The proof is completed by recalling that, under Assumption \ref{ass:qg} we have $\phi_t = \phir_{\x_t,1}$, and so
\begin{align*}
\phi_t(\y_1) - \phi_t(\x^*) & = Q_t(\x_t) + \R(\x_t) - Q_t(\x^*) - \R(\x^*) \\
& = \R(\x_t) - Q_t(\x^*) - \R(\x^*) \\
&\leq f(\x_t) - f(\x^*) + \R(\x_t) - \R(\x^*) = F(\x_t) - F(\x^*),
\end{align*}
where the inequality follows from Lemma \ref{lem:funie}.
\end{proof}

\subsection{Efficient implementation of first-order WPOs}\label{sec:fowpo:imp}
We now discuss concrete cases in which a first-order WPO could indeed be implemented very efficiently by leveraging the sparsity of $\x^*$.
\subsubsection{Symmetric subsets of $\reals^n$}\label{sec:fowpo:vecs}

\begin{assumption}\label{ass:symset}
$\R(\x)$ is the indicator function for a convex and closed subset  $\mK\subset\reals^n$ that is closed under permutation over coordinates. Furthermore, either $\mK$ is non-negative (i.e., $\mK\subseteq\reals^n_+$), or $\mK$ is closed under coordinate-wise sign-flips.
\end{assumption}

Note assumption \ref{ass:symset} covers $\ell_p$ balls in $\reals^n$ (for $p \geq 1$),  their restrictions to the non-negative orthant, and  the unit simplex. 

Let us introduce some useful notation. For any subset of indices $T\subseteq[n]$ and  $\x\in\reals^n$, we denote the restriction of $\x$ to $T$ as the vector $\x_T\in\reals^{|T|}$ such that for all $i\in\{1,\dots,|T|\}$, $\x_T(i) = \x(T_i)$, where $T_i$ denotes the ith element in $T$. We  denote the $n\times{}|T|$ matrix $\I_T$ which is obtained by deleting from the $n\times n$ identity matrix all columns which are not indexed by $T$. Thus, $\I_T\x_T$ maps $\x_T\in\reals^{|T|}$ to $\x\in\reals^n$ by setting each entry $i\in[n]$ in $\x$ to $[\x_T]_j$ in case the jth element in $T$ is $i$, and zero in case $i\notin{}T$.   Accordingly, we denote the restriction of $\mK$ to $T$ as the set $\mK_T = \{\x_T\in\reals^{|T|}~|~\I_T\x_T\in\mK\}$, and we let $\Pi_{\mK_T}[\cdot]$ denote the Euclidean projection onto $\mK_T$. In the following we let $\nnz(\x^*)$ denote the number of non-zero entries in the unique optimal solution to Problem \eqref{eq:optprob}.

\begin{lemma}\label{lem:wpoVecImp}
Suppose Assumption \ref{ass:symset} holds. Let $s\in[n]$ such that $s\geq \nnz(\x^*)$. Fix some iteration $i$ of Algorithm \ref{alg:wpo:1}, denote $\psi = \psi_{\y_i,\lambda_i}$ (see Eq. \eqref{eq:wpo:func})  and set
$\z := \y_i - \frac{1}{\lambda_i\tilde{\beta}}\nabla{}Q_t(\y_i)$.

If $\mK$ is  non-negative, letting $T\subset[n]$ be a set of  $s$ largest (signed) entries in $\z$, and defining $\z' = \I_T\Pi_{\mK_T}[\z_T]$, we have that $\w\in\arg\min_{\u\in\{\y_i,\z'\}}\psi(\u)$ satisfies Eq. \eqref{eq:wpo} (that is, $\w$ is a valid output of the first-order weak proximal oracle (Definition \ref{def:wpo}), when called with inputs $\y_i,\lambda_i$).

If $\mK$ is closed to coordinate-wise sign-flips,  letting $T\subset[n]$ be a set of $s$ largest in absolute value entries in $\z$, and defining $\z' = \I_T\Pi_{\mK_T}[\z_T]$, we have that $\w\in\arg\min_{\u\in\{\y_i,\z'\}}\psi(\u)$ satisfies Eq. \eqref{eq:wpo} 
\end{lemma}
\begin{proof}
Note that in both cases it suffices to prove that $\psi(\z') \leq \psi(\x^*)$.
First note that in both cases of the lemma $\z'$ is feasible w.r.t the indicator function $\R$, i.e., $\R(\z')=\R(\x^*)=0$. Thus, 
from the definition of $\psi$ (Eq. \eqref{eq:wpo:func})  we can see that $\psi(\z') \leq \psi(\x^*)$ if and only if $\Vert{\z'-\z}\Vert^2 \leq \Vert{\x^*-\z}\Vert^2$. Denote $\mK_s:=\mK\cap\{\x\in\reals^n~|~\nnz(\x)\leq s\}$ and note that under the assumption of the lemma, $\x^*\in\mK_s$. 

In both cases of the lemma it  follows from \cite{beck2016minimization} that $\z'$ satisfies $\z'\in\arg\min_{\p\in\mK_s}\Vert{\p-\z}\Vert$ (see Algorithm 2 in \cite{beck2016minimization} for the first case and Algorithm 3 for the second). Thus,  we indeed  have that $\Vert{\z'-\z}\Vert^2 \leq \Vert{\x^*-\z}\Vert^2$.
\end{proof}

Let us now discuss  concrete computational implications of Lemma \ref{lem:wpoVecImp} for our Problem \eqref{eq:optprob}.

\begin{corollary}\label{cor:wpo1:vec}
Suppose Assumption \ref{ass:symset} holds and that $\nnz(\x^*) \leq s$ for some $s << n$. Algorithm \ref{alg:wpo:1} admits an implementation based on Lemma \ref{lem:wpoVecImp} such that given the initial information $\x_t, \nabla{}f(\x_t), \nabla^2f(\x_t)\x_t$, each iteration requires:
\begin{enumerate}
\item
Single projection computation of  a $s$-dimensional vector onto a $s$-dimensional restriction of the set $\mK$ (e.g., if $\mK$ is an $\ell_p$ ball in $\reals^n$, then the projection is onto the $\ell_p$ ball in $\reals^s$).
\item
Single computation of the product of $\nabla^2f(\x_t)$ with a $s$-sparse vector\footnote{i.e., a vector in $\reals^n$ with at most $s$ non-zero entries}
\item
Additional $O(n\log{}s)$ time. 
\end{enumerate}
\end{corollary}
\begin{proof}
On each iteration  $i$ of Algorithm \ref{alg:wpo:1}, given the gradient $\nabla{}Q_t(\y_i)$, computing the point $\z'$ (as defined in Lemma  \ref{lem:wpoVecImp}) requires finding the $s$ largest entries in an $n$-dimensional array, which takes $O(n\log{}s)$ time and to project a $s$-dimensional vector onto the corresponding $s$-dimensional restriction of $\mK$. Since $\z'$ is $s$-sparse, a straightforward inspection reveals that the time to compute the next gradient direction $\nabla{}Q_t(\y_{i+1})$ is dominated by the time to multiply the Hessian $\nabla^2f(\x_t)$ with a $s$-sparse vector.
\end{proof}

\paragraph{Example: sparse optimization with an  $\ell_1$ constraint.} A classical example in which $\x^*$ is expected to be sparse is when $\R(\cdot)$ is the indicator function for a $\ell_1$ ball in $\reals^n$. Table \ref{table:fowpo:1} compares the complexity of Algorithm \ref{alg:wpo:1} to that of state-of-the-art methods such as accelerated gradient methods \cite{nesterov2013gradient, beck2009fast}  (which require in worst-case computing products of the Hessian with dense vectors to evaluate the gradient $\nabla{}Q_t(\cdot)$) and away-step Frank-Wolfe \cite{lacoste2015global} (which require products of the Hessian with only 1-sparse vectors but whose worst-case convergence rate scales at least linearly with the dimension $n$). We can see that in a meaningful regime of parameters, namely when $\tilde{\beta}/\alpha << (n/s)^2$, Algorithm \ref{alg:wpo:1} can indeed be considerably faster.

\begin{table*}\renewcommand{\arraystretch}{1.3}
{\footnotesize
\begin{center}
\begin{tabular}{ | p{16em} | p{8em} | p{5em}|} 
  \hline
  Algorithm &  time per iteration & \#{}iterations   \\\hline
  Algorithm  \ref{alg:wpo:1} & $sn$  &$\tilde{\beta}/\alpha$   \\ \hline
  Accelerated Gradient methods \cite{nesterov2013gradient, beck2009fast} & $n^2$  & $\sqrt{\tilde{\beta}/\alpha}$   \\ \hline
  Away-step Frank-Wolfe  \cite{lacoste2015global} & $n$  &  $n\tilde{\beta}/\alpha$  \\ \hline
\end{tabular}\caption{Comparison for solving subproblem of minimizing $\phi_t(\x)$ when $\R$ is an indicator function for a $\ell_1$ ball in $\reals^n$, when the Hessian $\nabla^2f(\x_t)$ is explicitly given and $\nnz(\x^*) \leq s << n$. Universal constants and logarithmic factors are suppressed.}\label{table:fowpo:1}
\end{center}}
\vskip -0.2in
\end{table*}\renewcommand{\arraystretch}{1}


\subsubsection{Spectrally-symmetric subsets of matrix spaces}\label{sec:matrices}

In the following, given positive integers $m,n$ and $\z\in\reals^{\min\{m,n\}}$, we let $\diag(\z)$ denote the $m\times n$ diagonal matrix with $\z$ as its main diagonal.

\begin{assumption}\label{ass:matrixset}
$\R(\cdot)$ is the indicator function for a closed and convex subset $\mM\subset\reals^{m\times n}$,  $\mM = \{\U\diag(\sigma)\V^{\top} ~|~ \U^{\top}\U = \I, ~\V^{\top}\V=\I,~\sigma\in\mK\}$, where $\mK$ is a convex and compact subset of $\reals^{\min\{m,n\}}$ that is non-negative (i.e., $\mK\subseteq\reals^{\min\{m,n\}}_+$) and closed under permutation over coordinates.  
\end{assumption}
Note Assumption \ref{ass:matrixset} covers Schatten $p$-norm balls in $\reals^{m\times n}$ (matrices in $\reals^{m\times n}$ with singular values lying in the corresponding  $\ell_p$ ball in $\reals^{\min\{m,n\}}$).

\begin{lemma}\label{lem:wpoMatImp}
Suppose Assumption \ref{ass:matrixset} holds and  write the SVD of $\x^*$ as $\x^*=\U^*\diag(\sigma^*)\V^{*\top}$ with $\sigma^*\in\mK$. Let $s\in[\min\{m,n\}]$ such that $\nnz(\sigma^*)\leq s$ (i.e., $\rank(\x^*) \leq s$). Fix some iteration $i$ of Algorithm \ref{alg:wpo:1} and denote $\psi = \psi_{\y_i,\lambda_i}$ (see Eq. \eqref{eq:wpo:func}. Denote  $\z = \y_i - \frac{1}{\lambda_i\tilde{\beta}}\nabla{}Q_t(\y_i)$ and let $\z=\U_{\z}\diag(\sigma_{\z})\V_{\z}^{\top}$ denote its SVD. Let $T\subset[\min\{m,n\}]$ be a set of  $s$ largest entries in $\sigma_{\z}$, set $\sigma_{\z}' = \I_T\Pi_{\mK_T}[\z_T]\in\mK$ and $\z'=\U_{\z}\diag(\sigma_{\z}')\V_{\z}^{\top}$. Then, $\w\in\arg\min_{\u\in\{\y_i,\z'\}}\psi(\u)$ satisfies Eq. \eqref{eq:wpo} (that is, $\w$ is a valid output of the first-order weak proximal oracle (Definition \ref{def:wpo}), when called with inputs $\y_i,\lambda_i$). 
\end{lemma}

\begin{proof}
As in the proof of Lemma \ref{lem:wpoVecImp}, it suffices to prove that $\Vert{\z'-\z}\Vert_F^2 \leq \Vert{\x^*-\z}\Vert_F^2$, where $\Vert{\cdot}\Vert_F$ denotes the Frobenius (Euclidean) norm for matrices. Similarly to the proof of Lemma \ref{lem:wpoVecImp}, using the results of \cite{beck2016minimization}, we have that
$\Vert{\sigma_{\z}' - \sigma_{\z}}\Vert \leq \Vert{\sigma^*-\sigma_{\z}}\Vert$.
This implies that
\begin{align*}
\Vert{\z' - \z}\Vert_F = \Vert{\U_{\z}\diag(\sigma_{\z}')\V_{\z}^{\top} - \z}\Vert_F \leq \Vert{\U_{\z}\diag(\sigma^*)\V_{\z}^{\top} - \z}\Vert_F.
\end{align*}

It thus remains to be shown that
\begin{align*}
\Vert{\U_{\z}\diag(\sigma^*)\V_{\z}^{\top} - \z}\Vert_F \leq \Vert{\U^*\diag(\sigma^*)\V^{*\top} - \z}\Vert_F.
\end{align*}
Since $\Vert{\U_{\z}\diag(\sigma^*)\V_{\z}^{\top}}\Vert_F = \Vert{\U^*\diag(\sigma^*)\V^{*\top}}\Vert_F$, it suffices to show that 
\begin{align*}
\langle{\U_{\z}\diag(\sigma^*)\V_{\z}^{\top}, \z}\rangle \geq \langle{\U^*\diag(\sigma^*)\V^{*\top}, \z}\rangle. 
\end{align*}
From von-Neumann's inequality for the singular values we have that, 
\begin{align*}
\langle{\U^*\diag(\sigma^*)\V^{*\top}, \z}\rangle &\leq \sum_{i=1}^{\min\{m,n\}}\sigma^*(i)\sigma_{\z}(i) = \langle{\U_{\z}\diag(\sigma^*)\V_{\z}^{\top}, \z}\rangle,
\end{align*}
which completes the proof.
\end{proof}


\begin{table*}\renewcommand{\arraystretch}{1.3}
{\footnotesize
\begin{center}
\begin{tabular}{ | p{16em} | p{8em} | p{5em}|} 
  \hline
  Algorithm &  SVD rank & \#{}iterations   \\\hline
  Algorithm  \ref{alg:wpo:1} & $s$   &$\tilde{\beta}/\alpha$  \\ \hline
  Accelerated Gradient methods \cite{nesterov2013gradient, beck2009fast} & $\min\{m,n\}$  & $\sqrt{\tilde{\beta}/\alpha}$  \\ \hline
  Frank-Wolfe \cite{jaggi2010simple} & $1$  &  $\tilde{\beta}\tau^2/\epsilon$  \\ \hline
\end{tabular}
\caption{Comparison for solving subproblem of minimizing $\phi_t$  when $\R$ is an indicator function for a nuclear norm ball in $\reals^{m\times n}$ of radius $\tau$  and $\rank(\x^*) \leq s << \min\{m,n\}$. Universal constants and logarithmic factors are suppressed. $\epsilon$ denotes the target accuracy.}\label{table:fowpo:2}
\end{center}}
\vskip -0.2in
\end{table*}\renewcommand{\arraystretch}{1}  

\begin{corollary}\label{cor:matix}
Suppose Assumption \ref{ass:matrixset} holds and that $\rank(\x^*) \leq s$ for some $s << \min\{m,n\}$. Then, Algorithm \ref{alg:wpo:1} admits an implementation based on Lemma \ref{lem:wpoMatImp} such that given the initial information: $\x_t, \nabla{}f(\x_t), \nabla^2f(\x_t)\x_t$, each iteration requires:
\begin{enumerate}
\item
Single rank-$s$ singular value decomposition (i.e., computing only the top $s$ components in the SVD) of a $m\times n$ real matrix. 
\item
Single projection computation of  a $s$-dimensional vector onto a $s$-dimensional restriction of the set $\mK$ (e.g., if $\mK$ is the restriction of an $\ell_p$ ball  in $\reals^{\min\{m,n\}}$ to the non-negative orthant, then the projection is onto the restriction of an $\ell_p$ ball  to the non-negative orthant in $\reals^s$).
\item
Single computation of the product of $\nabla^2f(\x_t)$ with a rank-$s$ $m\times n$ matrix (given by its SVD).
\item
Additional $O(smn)$ time. 
\end{enumerate}
\end{corollary}
\begin{proof}
On each iteration  $i$ of Algorithm \ref{alg:wpo:1}, given the gradient $\nabla{}Q_t(\y_i)$, computing the matrix $\z'$ (as defined in Lemma  \ref{lem:wpoMatImp}) requires finding the $s$ largest components in the SVD of $\z$ and to to project a $s$-dimensional vector onto the corresponding $s$-dimensional restriction of $\mK$. Since $\z'$ has rank at most $s$, a straightforward inspection reveals that the time to compute the next gradient direction $\nabla{}Q_t(\y_{i+1})$ is dominated by the time to multiply the Hessian $\nabla^2f(\x_t)$ with a rank-$s$ matrix given by its SVD.
\end{proof}

\paragraph{Example: low-rank optimization with a nuclear norm constraint.} A classical example in which $\x^*$ is expected to be a low-rank matrix is when $\R(\cdot)$ is the indicator function for a nuclear norm ball in $\reals^{m\times n}$: $\{\x\in\reals^{m\times n} ~|~ \sum_{i=1}^{\min\{m,n\}}\sigma_i(\x) \leq \tau\}$ for some $\tau > 0$, where $\sigma_i(\cdot)$ denotes the $i$th singular value. In this setting, a frequent computational bottleneck in solving the subproblems within Newton methods are high-rank SVD computations required to project onto the nuclear norm ball. Table \ref{table:fowpo:2} compares Algorithm \ref{alg:wpo:1} to  state-of-the-art methods such as accelerated gradient methods \cite{nesterov2013gradient, beck2009fast}  (which require in worst-case full rank SVD computations) and the Frank-Wolfe method \cite{jaggi2010simple} (which requires only rank-one SVDs, but converges only with a sublinear rate in this setting). For instance, when $m=n$, and considering that the time to compute a rank-$s$ SVD is often $\approx O(sn^2)$ in practice while the time to compute a full rank SVD is $O(n^3)$ in standard implementations, if $\tilde{\beta}/\alpha << (n/s)^2$ then Algorithm \ref{alg:wpo:1} may indeed be considerably faster.



\section{Implementing WPNOs via the Decomposition-Invariant Conditional Gradient Method}\label{sec:dicgwpo}
In this section we will be interested in the case that $\R$ is the indicator function for a convex and compact polytope in $\reals^n$. For the polytopes that will be relevant here, Euclidean projection is most often computationally difficult in high dimension, while there exists a much more efficient procedure for linear optimization over the polytope, and hence the interest in implementations of WPNOs (Definition \ref{def:2wpo}) based on  state-of-the-art conditional gradient methods for polyhedral sets \cite{garber2016linear, bashiri2017decomposition}. 
 In case $\x^*$ lies on a low-dimensional face of the polytope, these conditional gradient methods will enable to obtain linear convergence rates that scale only with the dimension of the optimal face (as opposed to the ambient dimension $n$ in standard linearly-converging conditional gradient methods \cite{garber2016linearly, lacoste2015global}).

Throughout this section we assume Assumption \ref{ass:qg} holds true, as well as the following assumption:
\begin{assumption}\label{ass:polytope}
$\R(\cdot)$ is an indicator function for a convex and compact polytope $\mP\subset\reals^n$ given in the form $\mP = \{\x\in\reals^n~|~\x\geq 0, ~\A\x=\b\}$, and the set of vertices of $\mP$, denoted by $\mV(\mP)$, are binary-valued, i.e., $\mV(\mP)\subset\{0,1\}^n$, and $|\mV(\mP)| \geq 2$.
\end{assumption}
Notable polytopes satisfying Assumption \ref{ass:polytope} include the \textit{flow polytope} of a combinatorial graph, the \textit{perfect matchings polytope} of a combinatorial bipartite graph,  the \textit{marginal polytope} associated with certain graphical models, and of course also the unit simplex \cite{garber2016linearly}.

We denote $n^* = \dim\mF^*+1$, where $\mF^*$ is the lowest dimensional face of $\mP$ containing $\x^*$ and $\dim\mF^*$ denotes its dimension. Note in particular that Carathéodory's theorem implies that $\x^*$ can be written as a convex combination of at most $n^*$ points from $\mV(\mP)$.

Throughout this section we fix some iteration $t$ of the iterations described in Eq. \eqref{eq:updatestep}, and accordingly we shall use the short notation $\phi_t = \phir_{\x_t,1}$,  $Q_t = Q_{\x_t}^r$ and  define the first-order smoothness parameter $\tilde{\beta}$ as in Eq. \eqref{eq:tbeta}.
\begin{table*}\renewcommand{\arraystretch}{1.3}
{\footnotesize
\begin{center}
\begin{tabular}{ | p{18em} | c|} 
  \hline
  Algorithm  & \#{}iterations scales with   \\\hline
  cond. grad. for general polytopes \cite{garber2016linearly, lacoste2015global} & $n$  \\ \hline
    DICG - standard analysis \cite{garber2016linear, bashiri2017decomposition}  & $n$  \\ \hline
 DICG  - our analysis & $n^*$    \\ \hline
\end{tabular}\caption{Comparison of conditional gradient methods for solving subproblem of minimizing $\phi_t$ when $\R$ is an indicator function for a polytope satisfying Assumption \ref{ass:polytope} and $n^*=\dim\mF^*+1 << n$.}\label{table:dicg}.
\end{center}}
\vskip -0.2in
\end{table*}\renewcommand{\arraystretch}{1}  
\begin{algorithm}
\begin{algorithmic}
\caption{WRPNO via the DICG Algorithm (see also \cite{garber2016linear})}\label{alg:PDICG}
\STATE Input: sequence of step-sizes $(\gamma_i)_{i\geq 1}\subset[0,1]$
\STATE $\y_1 \gets \arg\min_{\w\in\mV(\mP)}\langle{\w, \nabla{}Q_t(\x_t)}\rangle$
\FOR{$i=1,2,\dots$}
\STATE $\w_i^+ \gets\arg\min_{\w\in\mV(\mP)}\langle{\w,\nabla\phi_t(\y_i)}\rangle$ \COMMENT{compute Frank-Wolfe }
\STATE Define the vector $\widetilde{\nabla{}Q_t}(\y_i)\in\reals^n$ as follows:
\begin{align*}
\left[{\widetilde{\nabla{}Q_t}(\y_i)}\right]_j =  \left\{ \begin{array}{ll}
        \left[{\nabla{}Q_t(\y_i)}\right]_j  & \mbox{if $\y_i(j) > 0$}\\
        -\infty & \mbox{if $\y_i(j) = 0$}\end{array} \right.
\end{align*}
\STATE $\w_i^- \gets\arg\max_{\w\in\mV(\mP)}\langle{\w,\widetilde{\nabla{}Q_t}(\y_i)}\rangle$ \COMMENT{compute away direction }
\STATE $q_i \gets$ smallest  non-negative integer such that $2^{-q_i}  \leq \gamma_i$
\STATE $\lambda_i \gets 2^{-q_i}$
\IF{$\phi_t(\y_i + \lambda_i(\w_i^+-\w_i^{-})) < \phi_t(\y_i)$}
\STATE $\y_{i+1} \gets\y_i + \lambda_i(\w_i^+-\w_i^{-})$
\ELSE
\STATE $\y_{i+1} \gets \y_i$
\ENDIF
\ENDFOR
\end{algorithmic}
\end{algorithm}

\begin{theorem}\label{thm:dicg}
Suppose Assumptions \ref{ass:qg} and \ref{ass:polytope} hold. 
Set $\gamma_i = c_1^{1/2}(1-c_1)^{\frac{i-1}{2}}$ in Algorithm \ref{alg:PDICG}, where $c_1=\frac{\alpha}{32n^*\tilde{\beta}D^2}$. The sequence of iterates $(\y_i)_{i\geq 1}$ is feasible (w.r.t. $\mP$), and
\begin{align}\label{eq:thm:DICG:res}
\phi_t(\y_i) - \phi_t(\x^*) \leq \max\left\{\frac{\tilde{\beta}D^2}{2}\left(1{-\frac{\alpha}{32n^*\tilde{\beta}D^2}}\right)^{i-1}, ~ \frac{\beta_2}{3}\Vert{\x_t-\x^*}\Vert^3\right\}.
\end{align}
\end{theorem}
Thus, by running Algorithm \ref{alg:PDICG} for sufficiently-many iterations ($\tilde{O}(n^*\tilde{\beta}D^2/\alpha)$, where $\tilde{O}(\cdot)$ suppresses logarithmic factors), we can implement a WRPNO with parameter $C_{\mA} = O(\beta_2)$.

Before we prove the theorem we shall need the following lemma.
\begin{lemma}\label{lem:dicg:ie}
Fix an iteration  $i$ of Algorithm \ref{alg:PDICG} such that $\phi_t(\y_i) - \phi_t(\x^*) \geq \frac{\beta_2}{3}\Vert{\x_t-\x^*}\Vert^3$. Then,
\begin{align*}
\langle{\w_i^- - \w_i^+, \nabla{}Q_t(\y_i)}\rangle  \geq  \sqrt{\frac{\alpha}{4n^*}\left({\phi_t(\y_i) - \phi_t(\x^*)}\right)}.
\end{align*}
\end{lemma}
\begin{proof}
$\x^*$ can be written as $\x^* = \sum_{j=1}^{n^*}\delta_j\u_j$, where $(\delta_1,\dots,\delta_{n^*})$ is a distribution over $\{1,\dots,n^*\}$ and $\{\u_1,\dots,\u_{n^*}\}\subseteq\mV(\mP)$. Using Lemma 5 from \cite{bashiri2017decomposition}, it follows that $\y_i$ can be written as $\y_i = \sum_{j=1}^{n^*}(\delta_j-\Delta_j)\u_j + \sum_{j=1}^{n^*}\Delta_j\z$, where $\z\in\mP$, $0 \leq \Delta_j \leq \delta_j, j=1,\dots,n^*$, and the following bound holds:
\begin{align}\label{eq:dicg:1}
\sum_{j=1}^{n^*}\Delta_j \leq \sqrt{n^*}\Vert{\y_i-\x^*}\Vert.
\end{align}

This gives,
\begin{align*}
\langle{\x^*-\y_i, \nabla{}Q_t(\y_i)}\rangle = \sum_{j=1}^{n^*}\Delta_j\langle{\u_j - \z, \nabla{}Q_t(\y_i)}\rangle.
\end{align*}

Note that for all $j\in[n]$, $\y_i(j) =0$ implies that $\z(j)=0$. Using the definition of $\w_i^+,\w_i^-$ in Algorithm \ref{alg:PDICG}, we have the following inequalities: 
\begin{align*}
&\langle{\u_j - \w_i^+, \nabla{}Q_t(\y_i)}\rangle \geq 0, \quad j=1,\dots,n^*, \quad  \\
& \langle{\z - \w_i^-, \nabla{}Q_t(\y_i)}\rangle = \langle{\z - \w_i^-, \widetilde{\nabla{}Q_t}(\y_i)}\rangle \leq 0. 
\end{align*} 

Thus, using the convexity of $Q_t(\cdot)$, we have that

\begin{align*}
\sum_{j=1}^{n^*}\Delta_j\langle{\w_i^+ - \w_i^-, \nabla{}Q_t(\y_i)}\rangle &\leq \langle{\x^*-\y_i, \nabla{}Q_t(\y_i)}\rangle \leq Q_t(\x^*)-Q_t(\y_i) \\
&= \phi_t(\x^*) - \phi_t(\y_i),
\end{align*}
where the last equality holds since $\y_i$ is feasible and so $\R(\y_i) = \R(\x^*) = 0$.

Using the bound  in \eqref{eq:dicg:1}, Assumption \ref{ass:qg}, and Lemma \ref{lem:QGprox}, we have that
\begin{align*}
\sum_{j=1}^{n^*}\Delta_j &\leq \sqrt{n^*}\Vert{\y_i-\x^*}\Vert \leq \sqrt{\frac{2n^*}{\alpha}\left({F(\y_i) - F^*}\right)}\\
&\leq \sqrt{\frac{2n^*}{\alpha}\left({\phi_t(\y_i) - \phi_t(\x^*) + \frac{\beta_2}{3}\Vert{\x_t-\x^*}\Vert^3}\right)} \\
&\leq \sqrt{\frac{4n^*}{\alpha}\left({\phi_t(\y_i) - \phi_t(\x^*)}\right)},
\end{align*}
where the last inequality is due to the assumption of the lemma  that $\phi_t(\y_i) - \phi_t(\x^*) \geq \frac{\beta_2}{3}\Vert{\x_t-\x^*}\Vert^3$.

Combining the last two inequalities then gives,
\begin{align*}
\langle{\w_i^- - \w_i^+, \nabla{}Q_t(\y_i)}\rangle  \geq  \sqrt{\frac{\alpha}{4n^*}\left({\phi_t(\y_i) - \phi_t(\x^*)}\right)}.
\end{align*}

\end{proof}

\begin{proof}[Proof of Theorem \ref{thm:dicg}]
First, note that since $\mV(\mP)\subseteq\{0,1\}^n$, it follows that $D \geq 1$ and thus, the sequence of step-sizes listed in the theorem satisfies $(\gamma_i)_{i\geq 1}\subset[0,1]$ which, due to Lemma 1 in \cite{garber2016linear}, implies that the sequence $(\y_i)_{i\geq 1}$ is indeed feasible w.r.t. $\mP$.

Note also that by design, the sequence $(\phi_t(\y_i))_{i\geq 1}$ is monotone non-increasing. Thus, if on some iteration $i_0$ we have that $\phi_t(\y_{i_0}) - \phi_t(\x^*) \leq \frac{\beta_2}{3}\Vert{\x_t-\x^*}\Vert^3$ occurs for the first-time, then \eqref{eq:thm:DICG:res} indeed holds for all $i\geq i_0$.  Thus, in the remaining of the proof, in which we prove the first term inside the max in the RHS of \eqref{eq:thm:DICG:res}, we consider the iterations of Algorithm \ref{alg:PDICG} before iteration $i_0$.

Applying Lemma \ref{lem:dicg:ie} we have that,
\begin{align*}
\phi_t(\y_{i+1}) = Q_t(\y_{i+1}) &\leq Q_t(\y_i) + \lambda_i\langle{\w_i^+ - \w_i^-, \nabla{}Q_t(\y_i)}\rangle + \frac{\lambda_i^2\tilde{\beta}}{2}\Vert{\w_i^+-\w_i^-}\Vert^2 \\
&\leq \phi_t(\y_i)  -\lambda_i\sqrt{\frac{\alpha}{4n^*}\left({\phi_t(\y_i) - \phi_t(\x^*)}\right)} + \frac{\lambda_i^2\tilde{\beta}D^2}{2},
\end{align*}
where we have used the fact that both $\y_i,\y_{i+1}$ are guaranteed to be feasible and so $\phi_t(\y_{i+1}) = Q_t(\y_{i+1})$ and $\phi_t(\y_i)=Q_t(\y_i)$.

Observe that the choice of step-size $\lambda_i$ in Algorithm \ref{alg:PDICG} implies that $\frac{\gamma_i}{2} \leq \lambda_i \leq \gamma_i$.
This gives, using the choice of $\gamma_i$ stated in the lemma,
\begin{align}\label{eq:thm:DICG:1}
\phi_t(\y_{i+1}) &\leq \phi_t(\y_i)  - \frac{\gamma_i}{2}\sqrt{\frac{\alpha}{4n^*}\left({\phi_t(\y_i) - \phi_t(\x^*)}\right)}  + \frac{\gamma_i^2\tilde{\beta}{}D^2}{2} \nonumber \\
&\leq \phi_t(\y_i) - \frac{\alpha}{4\sqrt{32}n^*}\sqrt{\frac{1}{\tilde{\beta}D^2}}\left({1- \frac{\alpha}{32n^*{}\tilde{\beta}D^2}}\right)^{\frac{i-1}{2}}\sqrt{\phi_t(\y_i) - \phi_t(\x^*)} \nonumber \\
& ~~+ \frac{\alpha}{64n^*}\left({1- \frac{\alpha}{32n^*{}\tilde{\beta}D^2}}\right)^{i-1}.
\end{align}

We are now ready to prove by simple induction that: 
\begin{align*}
\forall i< i_0: \quad \phi_t(\y_i) - \phi_t(\x^*) \leq \frac{\tilde{\beta}D^2}{2}\left(1{-\frac{\alpha}{32n^*\tilde{\beta}D^2}}\right)^{i-1}.
\end{align*}

For the base case $i=1$, note that by the definition of $\y_1$ in the algorithm and the $\tilde{\beta}$-smoothness of $Q_t(\cdot)$ we have that, 
\begin{align*}
\phi_t(\y_1) - \phi_t(\x^*) &= Q_t(\y_1) - Q_t(\x^*) \\
&\leq Q_t(\x_t) - Q_t(\x^*) + \langle{\y_1-\x_t,\nabla{}Q_t(\x_t)}\rangle + \frac{\tilde{\beta}}{2}\Vert{\y_1-\x_t}\Vert^2 \\
&\leq  Q_t(\x_t) - Q_t(\x^*) + \langle{\x^*-\x_t,\nabla{}Q_t(\x_t)}\rangle + \frac{\tilde{\beta}D^2}{2} \leq \frac{\tilde{\beta}D^2}{2},
\end{align*}
where the last inequality is due to the convexity of $Q_t(\cdot)$.

For the induction step let us denote $h_i = \phi_t(\y_i)-\phi_t(\x^*)$ for all $i\geq 1$. Using the induction hypothesis and Eq. \eqref{eq:thm:DICG:1} we have that,
\begin{align*}
h_{i+1} & \leq h_i - \frac{\alpha}{4\sqrt{32}n^*}\sqrt{\frac{1}{\tilde{\beta}D^2}}\left({1- \frac{\alpha}{32n^*\tilde{\beta}D^2}}\right)^{\frac{i-1}{2}}\sqrt{h_i} + \frac{\alpha}{64n^*}\left({1- \frac{\alpha}{32n^*{}\tilde{\beta}D^2}}\right)^{i-1} \\
&\underset{(a)}\leq h_i - \frac{\sqrt{2}\alpha}{4\sqrt{32}n^*\tilde{\beta}D^2}h_i + \frac{\alpha}{64n^*}\left({1- \frac{\alpha}{32n^*\tilde{\beta}D^2}} \right)^{i-1} \\
&= h_i\left({1 - \frac{\sqrt{2}\alpha}{4\sqrt{32}n^*\tilde{\beta}D^2}}\right) + \frac{\alpha}{64n^*}\left({1- \frac{\alpha}{32n^*\tilde{\beta}D^2}} \right)^{i-1} \\
&\underset{(b)}\leq \frac{\tilde{\beta}D^2}{2}\left(1{-\frac{\alpha}{32n^*\tilde{\beta}D^2}}\right)^{i-1}\left({1 - \frac{\sqrt{2}\alpha}{4\sqrt{32}n^*\tilde{\beta}D^2}+ \frac{\alpha}{32n^*\tilde{\beta}D^2}}\right)\\ 
& = \frac{\tilde{\beta}D^2}{2}\left(1{-\frac{\alpha}{32n^*\tilde{\beta}D^2}}\right)^{i},
\end{align*}
where both (a) and (b) follow from the induction hypothesis. 

Thus, the induction holds.
\end{proof}

\section{Numerical Evidence}\label{sec:exp}
We turn to demonstrate the potential practical benefits of our theoretical results. We consider the following optimization model corresponding to the task of \textit{1-bit matrix completion} \cite{davenport20141}, which takes the form of minimizing a $\ell_2$-regularized logistic loss over a nuclear norm ball of matrices:
\begin{equation}
\label{eq:1bit-mc}
\min_{\x \in \mathbb{R}^{n_1 \times n_2} : \,\|\x\|_* \le \tau}
\left(
f(\x) := \sum_{(i,j)\in\Omega}
\log\bigl(1 + \exp(- y_{ij} \x_{ij})\bigr) + \frac{\rho}{2}\Vert{\x}\Vert_F^2
\right),
\end{equation}
where $\Omega \subseteq \{1,\dots,n_1\}\times\{1,\dots,n_2\}$
is the set of observed entries, and $y_{ij}\in\{-1,+1\}$ are the observed
1-bit measurements. Here $\|\cdot\|_*$ denotes the nuclear norm (sum of singular values) and $\tau > 0$ is the radius
of the feasible set. In all experiments we take $n_1 = n_2 = n $ and generate a rank-$r_{\sharp}$ random ground-truth
matrix $\x_\sharp \in \mathbb{R}^{n\times n}$ with random singular values of the form $0.1 + 3\cdot{}U[0,1]$. The observation set $\Omega$ is chosen uniformly at random without replacement from $\{1,\dots,n\} \times \{1,\dots,n\}$ with sampling ratio $0.5$. 
For each $(i,j)\in\Omega$ we observe a single binary outcome $y_{ij}\in\{-1,+1\}$ generated
according to the logistic 1-bit observation model:
\[
\mathbb{P}\bigl(y_{ij}=1 \,\big|\, {\x_{\sharp}}_{ij}\bigr)
= \frac{1}{1 + \exp(-{\x_{\sharp}}_{ij})},
\quad
\mathbb{P}\bigl(y_{ij}=-1 \,\big|\, {\x_{\sharp}}_{ij}\bigr)
= 1 - \mathbb{P}\bigl(y_{ij}=1 \,\big|\, {\x_{\sharp}}_{ij}\bigr).
\]
We set $\tau = \|\x_\sharp\|_*$ and $\rho = 0.1$.

We compare two implementations of the cubically-regularized proximal Newton method with unit step-size and with regularization parameter $1$ (i.e., iteratively minimizing $\phi_{\x,\eta}^r(\cdot)$, as defined in \eqref{eq:wrpno:func}, over the constraints with $\eta = 1$ and $\beta_2=1$). The first implementation is based on our \textit{first-order weak proximal oracle approach}  (Cubic{\_}WPO in the figures below) which uses Algorithm \ref{alg:wpo:1} to solve the subproblems (minimizations of  $\phi_{\x,\eta}^r(\cdot)$ over the nuclear norm ball),  as discussed in Section \ref{sec:matrices} (see Corollary \ref{cor:matix} and the example following it). In particular, this implementation requires only a single SVD computation of rank  $r_{\sharp} = \rank(\x_\sharp)$ per iteration of  Algorithm \ref{alg:wpo:1}, where for the sake of demonstration we assume $r_{\sharp}$ is known. We set a fixed step-size in  Algorithm \ref{alg:wpo:1}, $\lambda_i = 1/2$. 

As a baseline, we consider a second implementation of the cubic Newton method in which the subproblems are solved via FISTA with backtracking \cite{beck2009fast} (Cubic{\_}FISTA in the figures below). Note that FISTA requires a full-rank SVD computation in order to project onto the nuclear norm ball which is costlier than the low-rank SVD required by Algorithm \ref{alg:wpo:1}. 

For both implementations we use a maximum of $150$ iterations for solving each subproblem and we automatically stop once the distance between two consecutive iterates is at most $10^{-12}$. Additionally, in order to further demonstrate that the improved running times of our Cubic{\_}WPO implementation are due to the ability to leverage low-rank SVDs, we also consider a third implementation (Cubic{\_}WPO{\_}fullSVD in the figures below), which is the same as Cubic{\_}WPO with the sole difference that the low-rank SVDs are replaced with full-rank SVDs (as in Cubic{\_}FISTA).

Figure \ref{fig:exp} presents the convergence in function values for various values of $n, r_{\sharp}$, where each plot is the average of 20 i.i.d samples of the ground-truth matrix $\x_{\sharp}$. It can be clearly seen that in terms of the convergence of the Newton iterates (left panels in Figure \ref{fig:exp}), all methods have similar convergence, with Cubic{\_}FISTA having a slight advantage. In terms of the convergence as a function of the total number of iterations performed by the inner solver (center panels in Figure \ref{fig:exp}), we can see that, as expected, Cubic{\_}FISTA has a clear advantage which follows from the better dependence of FISTA on the condition number of the subproblems. However, when measuring the convergence in terms of the actual runtime  (right panels in Figure \ref{fig:exp}), we see that our Cubic{\_}WPO implementation has a substantial advantage. Moreover, from the slow convergence of the  Cubic{\_}WPO{\_}fullSVD implementation we clearly see that this improved runtime is indeed due to the ability to leverage the significantly more efficient low-rank SVD computations.



\begin{figure}[H]
\centering
\begin{subfigure}{0.32\textwidth}
    \includegraphics[width=\textwidth]{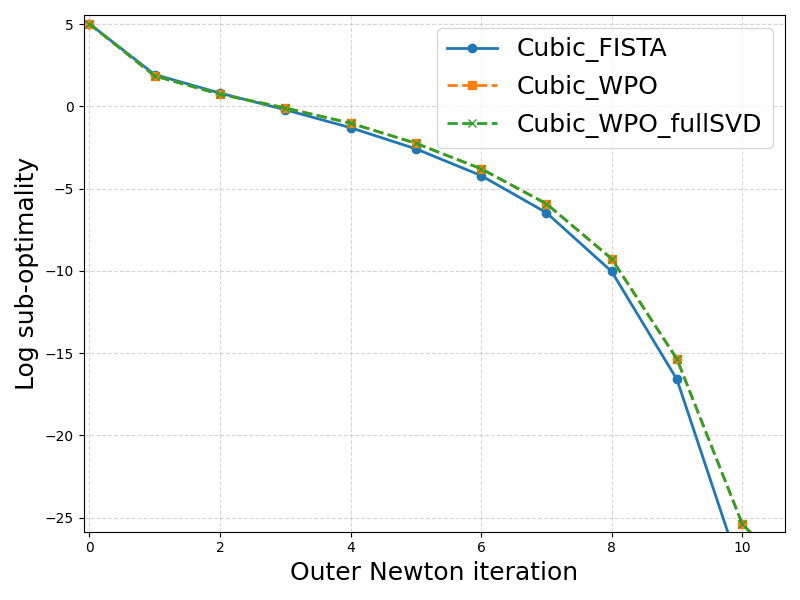}
    \caption*{}
    \label{fig:first}
\end{subfigure}
\hfill
\begin{subfigure}{0.32\textwidth}
    \includegraphics[width=\textwidth]{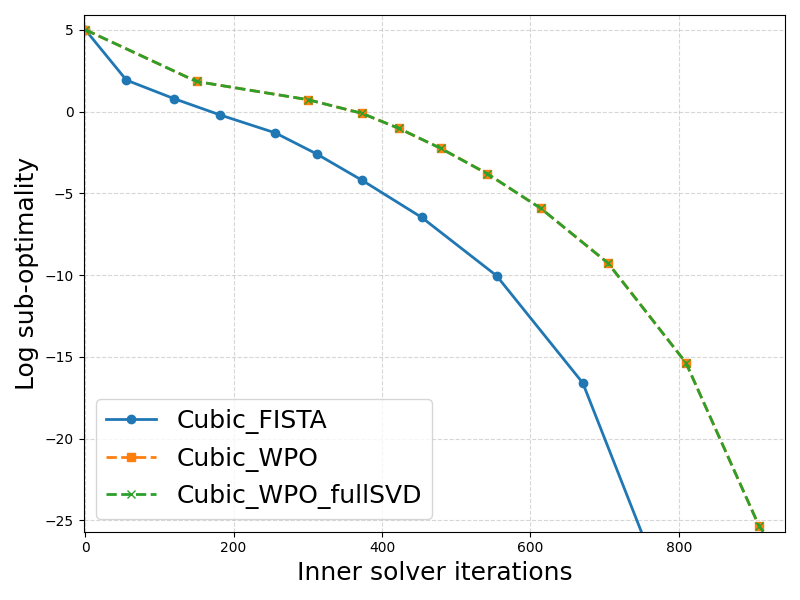}
    \caption{$n = 200, r_{\sharp} = 10$}
    \label{fig:second}
\end{subfigure}
\hfill
\begin{subfigure}{0.32\textwidth}
    \includegraphics[width=\textwidth]{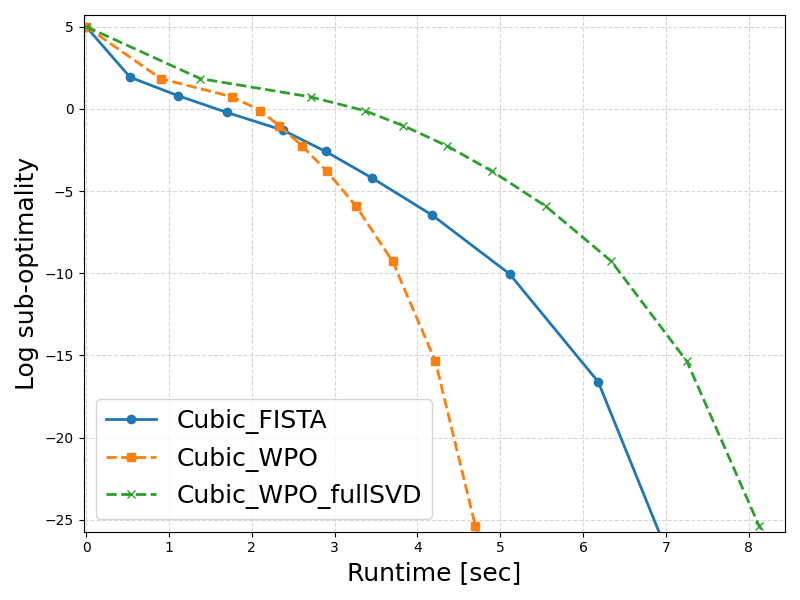}
    \caption*{}
    \label{fig:third}
\end{subfigure}

\begin{subfigure}{0.32\textwidth}
    \includegraphics[width=\textwidth]{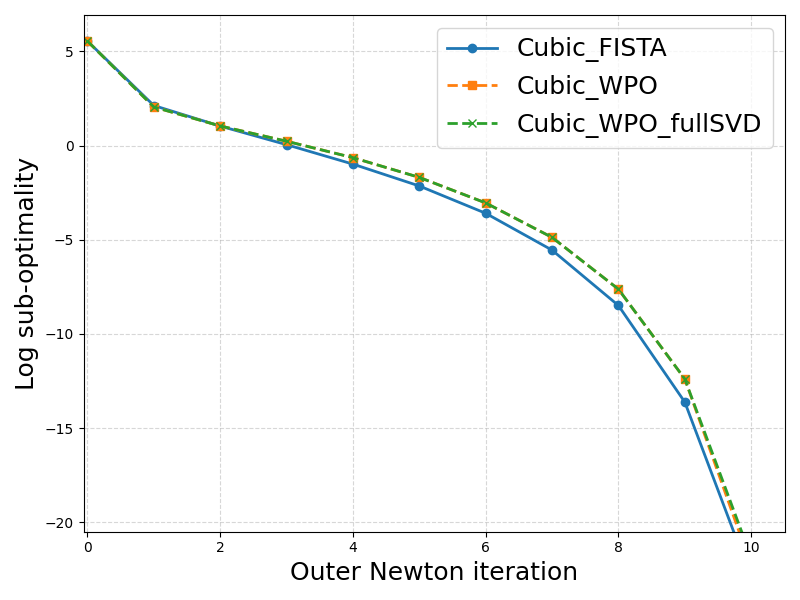}
    \caption*{}
    \label{fig:first}
\end{subfigure}
\hfill
\begin{subfigure}{0.32\textwidth}
    \includegraphics[width=\textwidth]{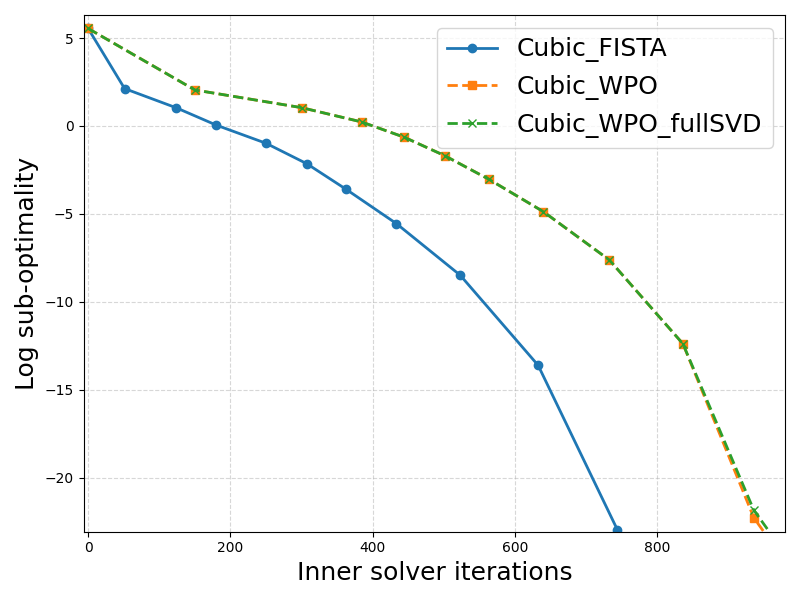}
    \caption{$n = 400, r_{\sharp} = 12$}
    \label{fig:second}
\end{subfigure}
\hfill
\begin{subfigure}{0.32\textwidth}
    \includegraphics[width=\textwidth]{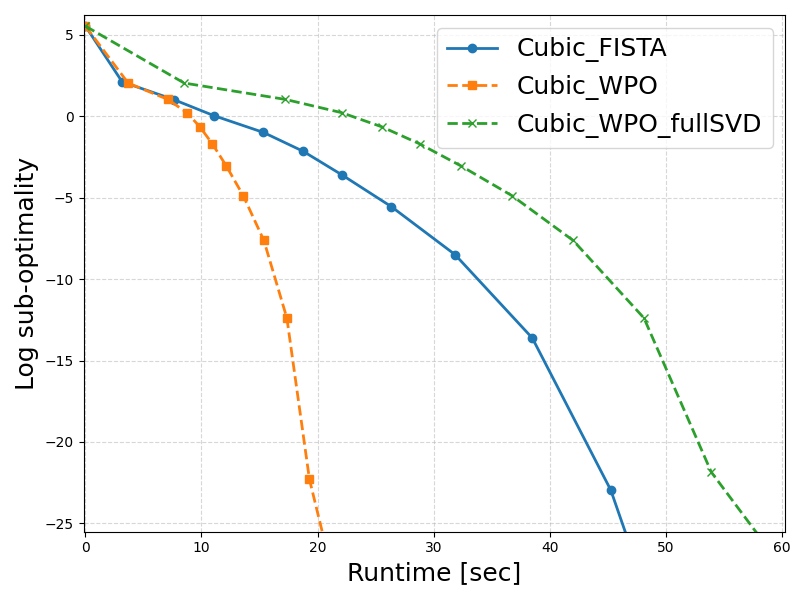}
    \caption*{}
    \label{fig:third}
\end{subfigure}
    
\begin{subfigure}{0.32\textwidth}
    \includegraphics[width=\textwidth]{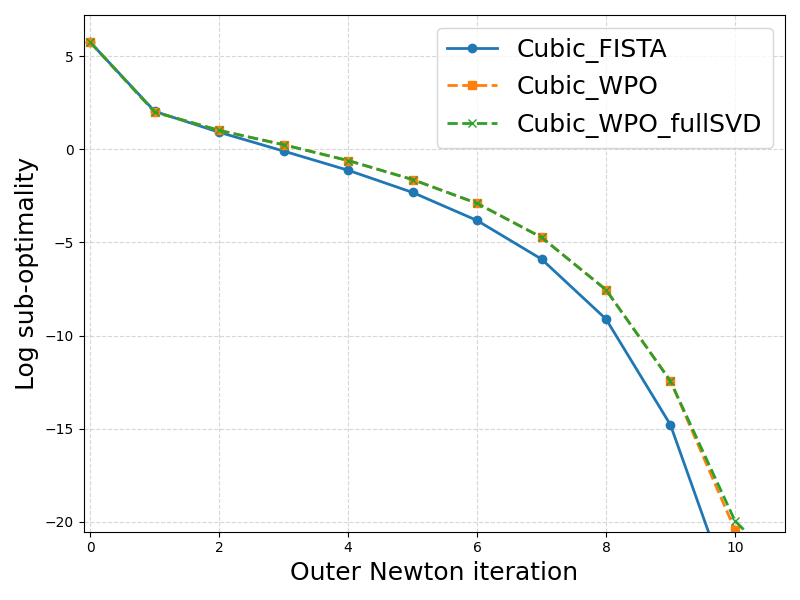}
    \caption*{}
    \label{fig:first}
\end{subfigure}
\hfill
\begin{subfigure}{0.32\textwidth}
    \includegraphics[width=\textwidth]{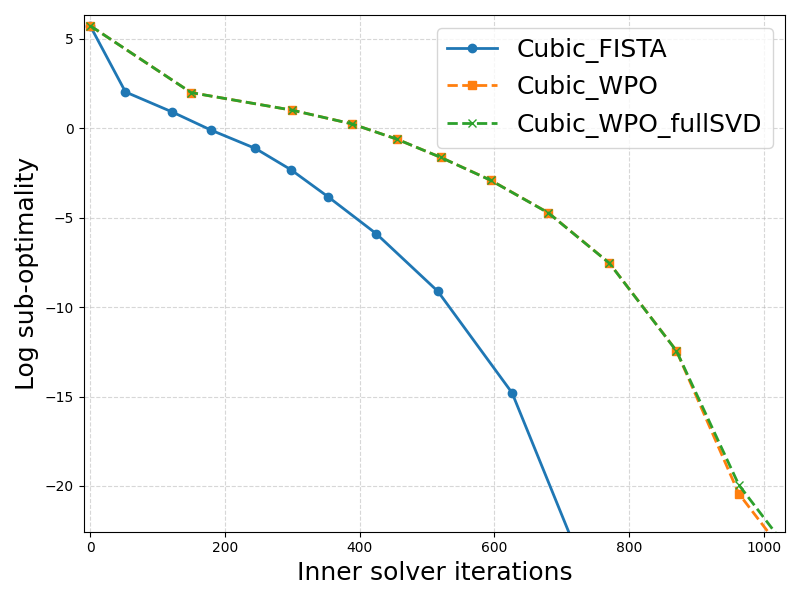}
    \caption{$n = 600, r_{\sharp} = 12$}
    \label{fig:second}
\end{subfigure}
\hfill
\begin{subfigure}{0.32\textwidth}
    \includegraphics[width=\textwidth]{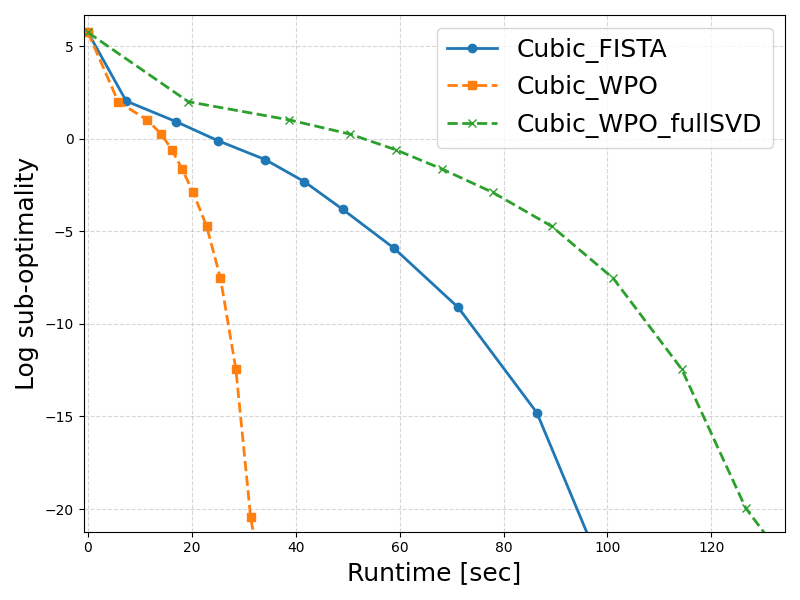}
    \caption*{}
    \label{fig:third}        
\end{subfigure}
        
\caption{Comparison of first-order WPO-based and FISTA-based  implementations of the cubically-regularized proximal Newton method. In the left and center panels the plots for Cubic{\_}WPO  and Cubic{\_}WPO{\_}fullSVD coincide.}
\label{fig:exp}
\end{figure}

\section{Acknowledgement}
This work was funded by the European Union (ERC,  ProFreeOpt, 101170791). Views and opinions expressed are however those of the author(s) only and do not necessarily reflect those of the European Union or the European Research Council Executive Agency. Neither the European Union nor the granting authority can be held responsible for them.

\appendix

\section{Proof of Lemma \ref{lem:distBound}}
We first restate the lemma and then prove it.
\begin{lemma}
Suppose Assumption \ref{ass:sc} holds and let $\x\in\mL_1$. Let $\v\in\dom(\R)$ be the output of a WPNO with approximation parameter $C_{\mA}$, when called with input $(\x,\eta)\in\dom(\R)\times\reals_+$. Then,
\begin{align*} 
\Vert{\v-\x^*}\Vert \leq \frac{\beta}{\alpha}\left({3+\frac{2}{\eta}}\right)\Vert{\x-\x^*}\Vert + \frac{\eta{}C_{\mA}}{\alpha}\Vert{\x-\x^*}\Vert^2.
\end{align*}
\end{lemma}
\begin{proof}
Define the function:
\begin{align*}
\theta(\w) := \langle{\w-\x^*,\nabla{}f(\x^*)}\rangle + \frac{\eta}{2}\langle{\w-\x^*,\nabla^2f(\x)(\w-\x^*)}\rangle + \R(\w).
\end{align*}
Note that due to the first-order optimality condition, $\x^*$ is a minimizer of $\theta(\cdot)$ over $\dom(\R)$. Furthermore, under Assumption \ref{ass:sc}, we have that $\theta(\w)$ is $\eta\alpha$-strongly convex. Thus,
\begin{align}\label{eq:lem:distBound:1}
\Vert{\v-\x^*}\Vert^2 &\leq \frac{2}{\eta\alpha}\left({\theta(\v) - \theta(\x^*)}\right) = \frac{2}{\eta\alpha}\left({\theta(\v) - \R(\x^*)}\right).
\end{align}
We have that,
\begin{align*}
\theta(\v) -\R(\x^*)&= \langle{\v-\x^*,\nabla{}f(\x^*)}\rangle + \frac{\eta}{2}\Vert{\v-\x^*}\Vert_{\x}^2  + \R(\v)- \R(\x^*)\\
&\underset{(a)}{\leq} \langle{\v-\x^*,\nabla{}f(\x)}\rangle + \beta\Vert{\v-\x^*}\Vert\Vert{\x-\x^*}\Vert + \frac{\eta}{2}\Vert{\v-\x^*}\Vert_{\x}^2  + \R(\v)- \R(\x^*)\\
&\underset{(b)}{=} \phi_{\x,\eta}(\v) - \phi_{\x,\eta}(\x^*) -\frac{\eta}{2}\Vert{\v-\x}\Vert_{\x}^2 + \frac{\eta}{2}\Vert{\x^*-\x}\Vert_{\x}^2  \\
&~~~+ \beta\Vert{\v-\x^*}\Vert\Vert{\x-\x^*}\Vert + \frac{\eta}{2}\Vert{\v-\x^*}\Vert_{\x}^2 \\
&\underset{(c)}{\leq}  \eta^2C_{\mA}\Vert{\x-\x^*}\Vert^3\\
&~~~-\frac{\eta}{2}\Vert{\v-\x}\Vert_{\x}^2 + \frac{\eta}{2}\Vert{\x^*-\x}\Vert_{\x}^2  + \beta\Vert{\v-\x^*}\Vert\Vert{\x-\x^*}\Vert + \frac{\eta}{2}\Vert{\v-\x^*}\Vert_{\x}^2,
\end{align*}
where (a) holds since, under Assumption \ref{ass:sc}, $f$ is $\beta$-smooth over the level set $\mL_1$, (b) holds by plugging-in the definition of $\phi_{\x,\eta}$ (Definition \ref{def:2wpo}), and (c) holds since $\v$ is the output of the WPNO.

Note that,
\begin{align*}
\Vert{\v-\x^*}\Vert_{\x}^2 - \Vert{\v-\x}\Vert_{\x}^2 &= \left({\Vert{\v-\x^*}\Vert_{\x} + \Vert{\v-\x}\Vert_{\x}}\right) \left({\Vert{\v-\x^*}\Vert_{\x} - \Vert{\v-\x}\Vert_{\x}}\right)\\
&\underset{(a)}{\leq}\left({2\Vert{\v-\x^*}\Vert_{\x} + \Vert{\x-\x^*}\Vert_{\x}}\right)\Vert{\x-\x^*}\Vert_{\x} \\
&= \Vert{\x-\x^*}\Vert_{\x}^2 + 2\Vert{\x-\x^*}\Vert_{\x}\Vert{\v-\x^*}\Vert_{\x}, \\
&\underset{(b)}{\leq} \Vert{\x-\x^*}\Vert_{\x}^2 + 2\beta\Vert{\x-\x^*}\Vert\Vert{\v-\x^*}\Vert,
\end{align*}
where (a) follows from using the triangle inequality twice, and (b) follows since, under Assumption \ref{ass:sc}, $\nabla^2f(\x) \preceq \beta\I$, and so for any $\z\in\E$, $\Vert{\z}\Vert_{\x} \leq \sqrt{\beta}\Vert{\z}\Vert$.

Plugging into the previous inequality, we have that
\begin{align}\label{eq:lem:distBound:2}
\theta(\v) - \R(\x^*) &\leq \eta^2C_{\mA}\Vert{\x-\x^*}\Vert^3 + \eta\Vert{\x-\x^*}\Vert_{\x}^2 + \left({\eta\beta+\beta}\right)\Vert{\x-\x^*}\Vert\Vert{\v-\x^*}\Vert \nonumber  \\
&\leq \eta^2C_{\mA}\Vert{\x-\x^*}\Vert^3+  \eta\beta\Vert{\x-\x^*}\Vert^2 + \beta(\eta+1)\Vert{\x-\x^*}\Vert\Vert{\v-\x^*}\Vert,
\end{align}
where in the last inequality we used again the fact that $\Vert{\z}\Vert_{\x} \leq \sqrt{\beta}\Vert{\z}\Vert$ for any $\z\in\E$.

Combining Eq. \eqref{eq:lem:distBound:1} and Eq. \eqref{eq:lem:distBound:2} we obtain,
\begin{align}\label{eq:lem:distBound:3}
\Vert{\v-\x^*}\Vert^2 \leq \frac{2\eta{}C_{\mA}}{\alpha}\Vert{\x-\x^*}\Vert^3 + \frac{2\beta}{\alpha}\Vert{\x-\x^*}\Vert^2 + \frac{2\beta(\eta+1)}{\alpha\eta}\Vert{\x-\x^*}\Vert\Vert{\v-\x^*}\Vert.
\end{align}
Let us consider now two cases. First, the case that $\Vert{\v-\x^*}\Vert \leq M\Vert{\x-\x^*}\Vert$ for some $M>0$ to be determined soon. In the complementing case, $\Vert{\v-\x^*}\Vert > M\Vert{\x-\x^*}\Vert$, we have from \eqref{eq:lem:distBound:3} that
\begin{align*}
\Vert{\v-\x^*}\Vert &\leq  \frac{2\eta{}C_{\mA}}{M\alpha}\Vert{\x-\x^*}\Vert^2 + \left({\frac{2\beta}{M\alpha} + \frac{2\beta}{\alpha} + \frac{2\beta}{\alpha\eta}}\right)\Vert{\x-\x^*}\Vert.
\end{align*}
Since $\beta/\alpha \geq 1$, setting $M=2$ gives that in either one of the cases,
\begin{align*} 
\Vert{\v-\x^*}\Vert \leq \frac{\beta}{\alpha}\left({3+\frac{2}{\eta}}\right)\Vert{\x-\x^*}\Vert + \frac{\eta{}C_{\mA}}{\alpha}\Vert{\x-\x^*}\Vert^2.
\end{align*}
\end{proof}

\bibliography{bibs.bib}

\begin{thebibliography}{10}

\bibitem{bashiri2017decomposition}
Mohammad~Ali Bashiri and Xinhua Zhang.
\newblock Decomposition-invariant conditional gradient for general polytopes
  with line search.
\newblock {\em Advances in neural information processing systems}, 30, 2017.

\bibitem{beck2016minimization}
Amir Beck and Nadav Hallak.
\newblock On the minimization over sparse symmetric sets: projections,
  optimality conditions, and algorithms.
\newblock {\em Mathematics of Operations Research}, 41(1):196--223, 2016.

\bibitem{beck2009fast}
Amir Beck and Marc Teboulle.
\newblock A fast iterative shrinkage-thresholding algorithm for linear inverse
  problems.
\newblock {\em SIAM journal on imaging sciences}, 2(1):183--202, 2009.

\bibitem{bertsekas1982projected}
Dimitri~P Bertsekas.
\newblock Projected newton methods for optimization problems with simple
  constraints.
\newblock {\em SIAM Journal on control and Optimization}, 20(2):221--246, 1982.

\bibitem{boyd2004convex}
Stephen Boyd.
\newblock Convex optimization.
\newblock {\em Cambridge UP}, 2004.

\bibitem{byrd2016inexact}
Richard~H Byrd, Jorge Nocedal, and Figen Oztoprak.
\newblock An inexact successive quadratic approximation method for l-1
  regularized optimization.
\newblock {\em Mathematical Programming}, 157(2):375--396, 2016.

\bibitem{carderera2020second}
Alejandro Carderera and Sebastian Pokutta.
\newblock Second-order conditional gradient sliding.
\newblock {\em arXiv preprint arXiv:2002.08907}, 2020.

\bibitem{davenport20141}
Mark~A Davenport, Yaniv Plan, Ewout Van Den~Berg, and Mary Wootters.
\newblock 1-bit matrix completion.
\newblock {\em Information and Inference: A Journal of the IMA}, 3(3):189--223,
  2014.

\bibitem{doikov2020convex}
Nikita Doikov and Yurii Nesterov.
\newblock Convex optimization based on global lower second-order models.
\newblock {\em Advances in Neural Information Processing Systems},
  33:16546--16556, 2020.

\bibitem{garber2016linearly}
Dan Garber and Elad Hazan.
\newblock A linearly convergent variant of the conditional gradient algorithm
  under strong convexity, with applications to online and stochastic
  optimization.
\newblock {\em SIAM Journal on Optimization}, 26(3):1493--1528, 2016.

\bibitem{pmlr-v89-garber19a}
Dan Garber and Atara Kaplan.
\newblock Fast stochastic algorithms for low-rank and nonsmooth matrix
  problems.
\newblock In Kamalika Chaudhuri and Masashi Sugiyama, editors, {\em Proceedings
  of the Twenty-Second International Conference on Artificial Intelligence and
  Statistics}, volume~89 of {\em Proceedings of Machine Learning Research},
  pages 286--294. PMLR, 16--18 Apr 2019.

\bibitem{garber2023faster}
Dan Garber, Tsur Livney, and Shoham Sabach.
\newblock Faster projection-free augmented lagrangian methods via weak proximal
  oracle.
\newblock In {\em International Conference on Artificial Intelligence and
  Statistics}, pages 7213--7238. PMLR, 2023.

\bibitem{garber2016linear}
Dan Garber and Ofer Meshi.
\newblock Linear-memory and decomposition-invariant linearly convergent
  conditional gradient algorithm for structured polytopes.
\newblock {\em Advances in neural information processing systems}, 29, 2016.

\bibitem{jaggi2010simple}
Martin Jaggi, Marek Sulovsk, et~al.
\newblock A simple algorithm for nuclear norm regularized problems.
\newblock In {\em Proceedings of the 27th international conference on machine
  learning (ICML-10)}, pages 471--478, 2010.

\bibitem{lacoste2015global}
Simon Lacoste-Julien and Martin Jaggi.
\newblock On the global linear convergence of frank-wolfe optimization
  variants.
\newblock {\em Advances in neural information processing systems}, 28, 2015.

\bibitem{lee2014proximal}
Jason~D Lee, Yuekai Sun, and Michael~A Saunders.
\newblock Proximal newton-type methods for minimizing composite functions.
\newblock {\em SIAM Journal on Optimization}, 24(3):1420--1443, 2014.

\bibitem{li2017inexact}
Jinchao Li, Martin~S Andersen, and Lieven Vandenberghe.
\newblock Inexact proximal newton methods for self-concordant functions.
\newblock {\em Mathematical Methods of Operations Research}, 85(1):19--41,
  2017.

\bibitem{liu2022newton}
Deyi Liu, Volkan Cevher, and Quoc Tran-Dinh.
\newblock A newton frank--wolfe method for constrained self-concordant
  minimization.
\newblock {\em Journal of Global Optimization}, pages 1--27, 2022.

\bibitem{nesterov2008accelerating}
Yu~Nesterov.
\newblock Accelerating the cubic regularization of newton’s method on convex
  problems.
\newblock {\em Mathematical Programming}, 112(1):159--181, 2008.

\bibitem{nesterov2013gradient}
Yu~Nesterov.
\newblock Gradient methods for minimizing composite functions.
\newblock {\em Mathematical programming}, 140(1):125--161, 2013.

\bibitem{nesterov1983method}
Yurii Nesterov.
\newblock A method for solving the convex programming problem with convergence
  rate o (1/k2).
\newblock In {\em Dokl akad nauk Sssr}, volume 269, page 543, 1983.

\bibitem{nesterov2006cubicCon}
Yurii Nesterov.
\newblock Cubic regularization of newton's method for convex problems with
  constraints.
\newblock 2006.

\bibitem{nesterov2006cubic}
Yurii Nesterov and Boris~T Polyak.
\newblock Cubic regularization of newton method and its global performance.
\newblock {\em Mathematical programming}, 108(1):177--205, 2006.

\bibitem{patriksson1998cost}
Michael Patriksson.
\newblock Cost approximation: a unified framework of descent algorithms for
  nonlinear programs.
\newblock {\em SIAM Journal on Optimization}, 8(2):561--582, 1998.

\bibitem{schmidt201211}
Mark Schmidt, Dongmin Kim, and Suvrit Sra.
\newblock 11 projected newton-type methods in machine learning.
\newblock {\em Optimization for Machine Learning}, (1), 2012.

\end{thebibliography}
\bibliographystyle{plain}
\end{document}